\documentclass{article}[11]
\usepackage{amsmath, amssymb, color}

\newtheorem{thm}{Theorem}[section]
\newtheorem{prop}[thm]{Proposition}
\newtheorem{lem}[thm]{Lemma}
\newtheorem{cor}[thm]{Corollary}
\newtheorem{definition}[thm]{Definition}

\newenvironment{dfn}{\medskip\refstepcounter{thm}
\noindent{\bf Definition \thesection.\arabic{thm}\ }}{\medskip}

\def\eq#1{{\rm(\ref{#1})}}
\newenvironment{proof}[1][,]{\medskip\ifcat,#1
\noindent{{\it Proof}:\ }\else\noindent{\it Proof of #1.\ }\fi}
{\hfill$\square$\medskip}
\newenvironment{remark}[1][Remark]{\begin{trivlist}
\item[\hskip \labelsep {\bfseries #1}]}{\end{trivlist}}

\DeclareMathOperator\vol{vol}
\DeclareMathOperator\GG{G}

\DeclareMathOperator\GL{GL}
\DeclareMathOperator\SO{SO}
\DeclareMathOperator\U{U}
\DeclareMathOperator\SU{SU}
\DeclareMathOperator\Sp{Sp}

\DeclareMathOperator\Diff{Diff}
\DeclareMathOperator\Ric{Ric}
\DeclareMathOperator\Ree{Re}
\DeclareMathOperator\Imm{Im}
\DeclareMathOperator\Vol{Vol}
\DeclareMathOperator\TRL{\mathcal{T}}
\DeclareMathOperator\tnabla{\widetilde{\nabla}}
\def\d{{\rm d}}
\def\w{\wedge}
\DeclareMathOperator\tr{tr}
\def\C{\mathbb{C}}
\def\R{\mathbb{R}}

\begin{document}

\title{From Lagrangian to totally real geometry: coupled flows
and calibrations}

\author{Jason D. Lotay \\ j.lotay@ucl.ac.uk \and Tommaso Pacini \\tommaso.pacini@unito.it}

%\date{}

\maketitle

\begin{abstract}
 We show that the properties of Lagrangian mean curvature flow are a special case of a more general phenomenon, concerning couplings between geometric flows of the ambient space 
and of totally real submanifolds. Both flows are driven by ambient Ricci curvature or, in the non-K\"ahler case, by its analogues. To this end we explore the geometry of totally 
real submanifolds, defining (i) a new geometric flow in terms of the ambient canonical bundle, (ii) a modified volume functional, further studied in \cite{LotayPacini2}, which takes into account the totally real condition. We discuss short-time existence for our flow and show it couples well with the Streets--Tian symplectic curvature flow for almost K\"ahler manifolds.
We also discuss possible applications to  
Lagrangian submanifolds and calibrated geometry.
\end{abstract}

\section{Introduction}

Lagrangian mean curvature flow has received much attention in the past fifteen years. Various directions have been pursued: intrinsic issues (maximal time existence and singularity formation \textit{e.g.}~\cite{Neves}, solitons \textit{e.g.}~\cite{IJO, JLT, LotayNeves}), applications to symplectic geometry (\textit{e.g.}~symplectic diffeomorphisms \cite{MedosWang}) and connections to mirror symmetry (analogies with Hermitian Yang--Mills \cite{ThomasYau}, relations to Bridgeland stability \cite{Joy}). 

The key observation from which this line of research originated is the following: under appropriate assumptions, a submanifold which is initially Lagrangian will remain Lagrangian when evolving under mean curvature flow (MCF), cf.~Theorem \ref{LMCF.KE.thm}.

Although this fact has been known for some time, the starting point for this paper is that this result is, in some sense, rather surprising. Indeed, MCF is a purely Riemannian concept and thus can be applied to any submanifold in an attempt to deform it to a minimal one. As a result, there is a priori no reason to hope that it preserves special properties of the submanifold, 
especially if those properties originate within a different branch of geometry.

The Lagrangian condition is an excellent example, as it has no metric content: it belongs to the world of symplectic geometry. Given a symplectic manifold 
$(M^{2n},\overline{\omega})$, a submanifold $\iota:L^n\rightarrow M$ is \textit{Lagrangian} if $\omega:=\iota^*\overline{\omega}\equiv 0$. This notion plays a key 
role in mechanics, symplectic topology and mirror symmetry.

The secret to the above relationship between Lagrangians and MCF lies  in the ``appropriate assumptions''. The obvious assumption is the ambient manifold should be K\"ahler: this is the standard setting for interactions between Riemannian and symplectic geometry. It turns out however that this assumption, by itself, is not sufficient. A curvature assumption is also necessary: the above result is true only in K\"ahler--Einstein (KE) manifolds, cf.~Theorem \ref{LMCF.KRF.thm}.

The KE assumption is very strong. Nevertheless, it contains Ricci-flat K\"ahler and thus Calabi--Yau (CY) manifolds as a subclass, 
and this is one reason for interest in Lagrangian MCF. Indeed, CY manifolds and their Lagrangian submanifolds have attracted much attention in connection with string theory and for their role in mirror symmetry. There are also direct links to a natural class of  volume-minimizing submanifolds within the context of calibrated geometry.

\paragraph{Coupled flows.} It is an interesting question whether one can weaken the KE assumption. It is a much less well known fact due to Smoczyk \cite{Smo} that, if the ambient manifold is only K\"ahler and not necessarily KE, one should couple the MCF equation for the submanifolds with the K\"ahler--Ricci flow (KRF) on the ambient space; this system of equations yields a flow on submanifolds which preserves the Lagrangian condition. 

It is curious that Smoczyk's result has received little attention.  Consider the following. 
\begin{itemize}
\item In recent years there has been increasing evidence \cite{McCTopp,Muller} that coupling two geometric flows can lead each to exhibit better properties than it had by itself, 
both from the geometric and the analytic perspective. 
Analytic properties of the coupling of MCF with Ricci flow  have been studied in \cite{Lott,MMT}. 
\item Perelman's work has made it obvious that Ricci flow is currently the most interesting ambient geometric flow available. Recent 
work on the existence of K\"ahler--Einstein metrics and on the Minimal Model Program provides good motivation to study KRF, which additionally has been shown to exhibit some good long-time existence properties. 
\item For submanifolds, the obvious flow is MCF. However, it was perhaps most clearly pointed out by Oh \cite{Oh2} that, in a K\"ahler manifold, MCF of a Lagrangian $L$ can be viewed as being driven by the ambient Ricci curvature $\overline{\mbox{Ric}}$: the mean curvature vector $H$ on $L$ is equivalent to the 1-form $\xi:=\overline{\omega}(H,\cdot)$ on $L$, and $\d\xi=\overline{\mbox{Ric}}(J\cdot,\cdot)$ (where $J$ is the complex structure). 
\end{itemize}

The above should already convince us of the intrinsic value of Smoczyk's result. It also indicates why the result might be true. The variation of $\omega$ depends on
 $\overline{\mbox{Ric}}$, explaining the role of the geometric coupling: the two contributions of the Ricci curvature ultimately cancel each other out, 
 leaving $\omega$ unchanged.

We now add a further consideration. In the past few years several generalizations of KRF have been suggested, reducing the integrability assumptions on the initial geometric structures. It is important to question which of these provide the most promising avenues for further investigation. There are standard reasons to prefer one flow over another: geometric motivations for the definition, or a proof of short-time existence. We wish to suggest an addition to this list. 

\ 

\noindent\textit{Test}: Does the flow couple profitably with another? 

\ 

If so, it seems definitely worthwhile to investigate it further. Notice that this test may also help identify geometrically interesting lower-order perturbations of a given flow. If the flow is (weakly) parabolic, such terms do not affect the short-time existence theory: in this sense they are not analytically detectable. In this paper, our geometrically motivated lower-order terms turn out to play an important role even in the existence theory.

\paragraph{Totally real submanifolds.} Lagrangians are a special case of a much wider family of submanifolds called \textit{totally real submanifolds}: 
those which are ``maximally non-complex'', where ``maximal'' also refers to the dimension of the
submanifold. Totally real submanifolds are a key part in the proof that MCF preserves Lagrangians. However, the standard proof down-plays them, simply viewing them as possible degenerations of Lagrangians to be ruled out. 

More generally, totally real submanifolds seem to have received only sporadic interest, \textit{e.g.}~\cite{Bor,ChenOg}. Once again this is curious, considering that by definition they are the ``exact opposite'' of the most classical class of submanifolds: complex subvarieties. It may be that this lack of interest is due to several factors. 
\begin{itemize}
 \item The suspicion that the defining condition is simply too weak: it is an open condition in the space of immersions, so the class of totally real submanifolds is huge.
 \item Pseudo-holomorphic curves whose boundary is contained in a totally real submanifold constitute a well-defined elliptic problem, but to obtain good compactness properties for such curves one requires that the boundary lies in a Lagrangian. This property helps make Lagrangians and pseudo-holomorphic curves a key tool in symplectic topology.
 \item The most straightforward analogue of Smoczyk's result fails for general totally real submanifolds: specifically, if one couples MCF of a totally real, but not Lagrangian, submanifold with KRF, the initial values of $\omega:=\iota^*\overline{\omega}$ are usually not preserved.
 \item In some contexts, \textit{e.g.}~when working with homotopy classes of immersions, the difference between Lagrangian and totally real is irrelevant: 
this is a consequence of the validity of the ``h-principle'', cf.~\cite{HP} for an application. 
\end{itemize}

\paragraph{Results on totally real submanifolds.} 
In Sections \ref{s:totally_real}-\ref{s:Jvol} we try to counter the above impression by showing that totally real submanifolds carry interesting geometry which is hard to notice when one restricts to Lagrangians. Specifically, we demonstrate the following.
\begin{itemize}
 \item Totally real submanifolds $L$ in $(M,J)$ can be characterized in terms of the canonical bundle $K_M$. This leads to an intrinsic notion of volume of $L$, called the \textit{$J$-volume}, and to a natural $1$-form $\xi_J$ on $L$, which we call the \textit{Maslov form}.
 \item The gradient of the $J$-volume and the Maslov 1-form are linked by a key formula which, in the simplified K\"ahler case, takes the form
 $$\bar\omega(H_J,\cdot)_{|TL}=\xi_J,\ \ \ \mbox{d}\xi_{J|TL}=\overline{\mbox{Ric}}(J\cdot,\cdot)_{|TL}$$
 for any totally real $L$, where $H_J$ is an explicit vector field along $L$. 
 \item The above data defines two natural flows on totally real submanifolds: the gradient flow of the $J$-volume functional and the \textit{Maslov flow} (MF). These flows are in general distinct and, in the presence of a Riemannian structure, different also from MCF. However, we show in Section \ref{s:Jvol} that all three coincide for Lagrangians in 
 K\"ahler manifolds, recovering the standard theory.
\end{itemize}
We refer to these flows as ``canonical'', both because they are generated by the canonical bundle of $M$ and because of links to other aspects of the geometry of totally real submanifolds, studied in \cite{LotayPacini2}: in particular, our $J$-volume functional turns out to have interesting convexity properties with respect to an intrinsic notion of geodesics on the space of totally real submanifolds.

\paragraph{Results on coupled flows.} Having introduced these flows, we need to meet our own standards of ``what makes a new flow interesting''. Our main results concern the Maslov flow, which thus becomes the focus of the paper. 

Analytically, the main issue we address is short time existence. 
This turns out to be  a rather subtle question due to the degeneracies of the operators involved. 
We give several results in this direction, but the general picture remains open. Our proofs rely on Hamilton's version of the Nash--Moser inverse function theorem concerning operators satisfying an ``integrability condition''. In our case this condition is provided by the key formula introduced above: this is a new twist on the use of this formula, even in the classical case of Lagrangians in K\"ahler manifolds. This method thus emphasizes the link between the existence theory and preserved quantities.

Geometrically, the main point is the existence of interesting couplings. We explain how the Maslov flow interacts with another ``new entry'': \textit{symplectic curvature flow} (SCF), introduced by Streets--Tian in \cite{StrTian}. Our Theorem \ref{MSCF.thm} shows that the coupled system SCF with MF preserves not just the Lagrangian condition, but any given initial values of $\omega$ on any initial totally real submanifold. Once again, the main tool is the above mentioned key formula.

To explain this more carefully, however, we need another digression. The context for SCF is that of \textit{almost K\"ahler} manifolds. This class, which vastly generalizes that of K\"ahler manifolds, has been studied at least since the 1950s, but mainly by researchers interested in ``geometric structures with torsion'' \cite{ApDr,GrayHervella}.  The point of view we wish to press here is that the above results indicate that almost K\"ahler manifolds should also have a role to play in geometric analysis. In particular, we observe the following.
\begin{itemize}
 \item The K\"ahler condition is very strong: so strong that, by forcing several quantities to vanish, it ends up obfuscating some interesting geometry.
 \item Almost K\"ahler manifolds essentially coincide with the symplectic category: in the past decades symplectic topology has flourished, and now has a large amount of specific techniques at its disposal.
 \item Our Theorem \ref{MSCF.thm} expresses a new reason for interest in both SCF and in MF, thus in almost K\"ahler manifolds.
 \end{itemize}
 
\paragraph{Conclusions on Lagrangian MCF.}We can now return to our original question: how does MCF manage to preserve Lagrangian submanifolds?

Our answer, simply put, is that this is purely accidental. More explicitly, at the
 end of Section \ref{ss:existence_lagr_mcf} we emphasize that it is the 
 combination of the Lagrangian and K\"ahler assumptions which force the result 
 to hold simply by deleting all terms that would prevent it from being true. 
 
 The main goal of this paper is to show that these terms encode interesting aspects of the submanifolds' geometry. 
It follows that Lagrangian MCF is not really about the Riemannian volume. It is about a different geometry 
encoded by the Maslov form, which coincides with MCF only in restricted situations. By replacing Lagrangians with totally reals 
we bring to light, in this context, the role of the complex structure.

\paragraph{Results within Calabi--Yau manifolds.} Sections \ref{s:maslov_CY}-\ref{s:calibration} of this paper are dedicated to the case of Calabi--Yau manifolds. 
Here, there was already a classical notion of Maslov form for Lagrangians. We show that our own notion is a generalization of that one  to the much 
wider context of totally real submanifolds in almost complex manifolds. We also provide a link with calibrated geometry by showing that the critical points of our flows 
may be interpreted as a generalized version of calibrated submanifolds, called \emph{special totally real} (STR). This implies that they are automatically absolute 
$J$-volume minimizers.

\paragraph{Open problems.} Several issues seem worthy of further investigation.

\smallskip
\textit{Stronger existence results.} Our current result, cf. Corollary \ref{Lag.MF.exist.cor}, concerns the short-time existence of solutions to the Maslov flow under two assumptions: (i) the initial submanifold is Lagrangian, (ii) the  almost K\"ahler manifold satisfies a certain Einstein-like constraint (formulated in terms of the Chern connection).  More generally, cf.~Theorem \ref{thm:existence_projectedeq}, we can prove existence for any initial totally real submanifold and any almost K\"ahler manifold if we perturb the flow slightly (still preserving the desired ``integrability condition'', but sacrificing the geometric motivation of the flow). The main open question is whether the Maslov flow exists for any totally real submanifold, under coupling with SCF (so as to remove the Einstein constraint).

\smallskip
\textit{Applications to Lagrangian geometry.} One may speculate on the extent to which our results on totally real submanifolds may be relevant to the study of Lagrangians. 
For example, in Section \ref{ss:CYdevelopments} we consider potential applications to the study of minimal Lagrangians in Calabi--Yau manifolds and relations to ideas introduced in \cite{Don,ThomasYau}. The negative KE direction is pursued further in \cite{LotayPacini2, LotayPacini3}. In this context, our Proposition \ref{min.Lag.lim.prop} shows that replacing the standard volume functional with the $J$-volume serves to ``weed out'' any non-Lagrangian critical points, thus helping to focus only on minimal Lagrangians.

\smallskip
\textit{$\GG_2$ geometry and gauge theory.} Another direction concerns possible analogues within $\GG_2$ geometry, where \textit{coassociative} submanifolds play the role of Lagrangian submanifolds. This analogy actually provided the initial motivation for this paper.  One can also propose coupled flows in the context of $\GG_2$ gauge theory.

\smallskip
Clearly there are possible analogues also in the context of Legendrian submanifolds within Sasaki manifolds, cf.~\cite{Kawai}, \cite{SmoLeg} for results in this direction.

Some of the above is work in progress by the authors.

\paragraph{Relation to previous literature.} 
Our point of view on Lagrangian MCF is perhaps most closely related to that originally put forth by Oh \cite{Oh2}, but goes far beyond that. There is also some intersection with \cite{SmoWang}, which generalizes previous results concerning Lagrangian MCF to the almost K\"ahler setting. However, it focuses only on Lagrangians and pre-dates the work of Streets--Tian, so it does not make use of any coupling with symplectic curvature flow. 

Our main debt is towards Borrelli \cite{Bor}, who first extended the notion of Maslov class from Lagrangians to totally real submanifolds. He also introduces the notion of $J$-volume and of special totally real submanifolds; accordingly, we adopt his terminology. The setting he works in however is mostly $\C^n$, sometimes K\"ahler manifolds, and he does not pursue geometric flows.

\paragraph{Summary.} We conclude this introduction by briefly 
summarising the contents of the article.
\begin{itemize}
\item In $\S$\ref{review} we review the standard viewpoint on Lagrangian mean curvature flow.  We use this as an opportunity to highlight key points and  introduce notation and basic geometric notions which we use throughout the article.  
\item In $\S$\ref{s:totally_real} we study totally real submanifolds in the general setting  
 and define the $J$-volume and Maslov form.  
\item In $\S$\ref{s:hermitian} we specialise to almost Hermitian manifolds
 where we relate our $J$-volume and Maslov form to the ambient Riemannian 
 geometry. We identify the Maslov form with a vector field which provides the natural substitute for the mean curvature vector field, and make further observations in the almost K\"ahler and K\"ahler  settings.  
\item In $\S$\ref{s:Jvol} we investigate the properties of the critical points of the $J$-volume, using the first variation formula computed in \cite{LotayPacini2}. This formula also serves to introduce the $J$-mean curvature flow, \textit{i.e.}~the negative gradient flow of the $J$-volume.
\item In $\S$\ref{s:MF} we study the Maslov flow in almost K\"ahler manifolds and show that it couples well with symplectic curvature flow, cf.~Theorem \ref{MSCF.thm}. In K\"ahler manifolds, the Maslov flow coincides with the $J$-mean curvature flow. 
We examine the long-time
behaviour of the Maslov flow in the K\"ahler--Einstein setting and relate it to other known flows for Lagrangians.
\item In $\S$\ref{s:shorttime} we investigate the short-time existence and uniqueness of the Maslov flow by computing the symbol of the flow operator and applying theory due to Hamilton \cite{Hamilton} that invokes the Nash--Moser inverse function theorem.  
\item In $\S$\ref{s:maslov_CY} we review the classical notion of Maslov form in Calabi--Yau manifolds and relate it to our study, generalizing the known 
concept and giving a further characterisation for critical points of the $J$-volume functional.
\item In $\S$\ref{s:calibration} we generalize the calibrated geometry of special Lagrangians in Calabi--Yau manifolds to 
totally real submanifolds, thus providing further results for critical points of our flows.  We also relate our results here to stability and moduli 
space questions for special Lagrangians, and to graphs of maps between 
almost complex manifolds.
\end{itemize}

\ 

\textit{Thanks to} Dominic Joyce in particular for invaluable discussions.  We would also like to thank Alberto Abbondandolo, Antonio Ache, Leonardo Biliotti, 
Jonny Evans and Luigi Vezzoni for useful conversations and Vincent Borrelli for providing a copy of his PhD thesis. We further thank Denis 
Auroux, Johannes Nordstr\"om, Jake Solomon and Richard Thomas for comments on a preliminary version of this paper. 

JDL was partially supported by an EPSRC Career Acceleration Fellowship. TP is grateful to Oxford University for its hospitality during some stages of this project and to a Marie Curie reintegration grant for funding.

\section{Review of Lagrangian MCF}\label{review}

Given any Riemannian manifold $M$ and any immersion $\iota:L\rightarrow M$ of an oriented manifold $L$, we say that a one-parameter family of immersions $\iota_t:L\rightarrow M$ such that $\iota_0=\iota$ satisfies
mean curvature flow (MCF) if 
\begin{equation}\label{MCF.eq}
\frac{\partial\iota_t}{\partial t}=H_t,
\end{equation}
where $H_t$ is the mean curvature vector of the immersion $\iota_t$.  
When $L$ is compact, this flow is the negative gradient flow of the volume functional; stationary points for MCF are minimal submanifolds.  

Suppose $M$ also has a symplectic structure (so $M$ is $2n$-dimensional), which is given by a closed non-degenerate 2-form $\overline{\omega}$, and that $\iota:L\rightarrow M$ is
\emph{Lagrangian}; i.e.~that $L$ is $n$-dimensional  (half the dimension of $M$) and $\iota^*\overline{\omega}=0$. If $L$ moves by MCF, then in general 
$\iota_t:L\rightarrow M$ is no longer Lagrangian.  However, there are 
situations where MCF does indeed preserve the 
Lagrangian condition, leading to the notion of \textit{Lagrangian mean curvature flow}.  

The aim of this section is to review the fact that Lagrangians are preserved by MCF in K\"ahler-Einstein manifolds, and 
also in K\"ahler manifolds when coupled with K\"ahler--Ricci flow. This result is originally due to Smoczyk \cite{Smo}.  However, the key calculation goes at least 
as far back as \cite{Daz} and the ideas behind deforming Lagrangians in KE manifolds are certainly discussed in \cite{Oh, Oh2}.

For this section we suppose that $M$ is a K\"ahler manifold, \textit{i.e.} a complex $n$-manifold with 
 a Riemannian metric $\overline{g}$ with Levi-Civita connection $\overline{\nabla}$, 
an orthogonal complex structure $J$ and
a 2-form $\overline{\omega}$ given by $\overline{\omega}(X,Y)=\overline{g}(JX,Y)$ which is nondegenerate and satisfies $\d\overline{\omega}=0$.
We let $\overline{R}$ and $\overline{\Ric}$ denote the Riemann curvature tensor and Ricci tensor of $\overline{g}$ and define
$$
\overline{\rho}(X,Y)=\overline{\Ric}(JX,Y).
$$
The 2-form $\overline{\rho}$ is the natural way to view the Ricci tensor as a $(1,1)$-form on $M$ and 
we have that the first Chern class of $M$ is given by $2\pi c_1(M)=[\overline{\rho}]$.    We may easily express the K\"ahler--Einstein (KE) condition as $\overline{\rho}=\lambda\overline{\omega}$ for a constant $\lambda$, and we can write
 K\"ahler--Ricci flow (KRF) as:
\begin{equation}\label{KRF.eq}
\frac{\partial\overline{\omega}_t}{\partial t}=-\overline{\rho}_t.
\end{equation}
An equivalent definition of KRF
is to have the metric evolve by Ricci flow and define the K\"ahler form to be compatible with the evolving metric and fixed complex structure.

\begin{remark}
KE manifolds are solitons for KRF; that is, they are either stationary, shrinking or expanding solutions depending on 
whether $\lambda$ is $0$, negative or positive.  One can also defined \emph{normalized} K\"ahler--Ricci flow by 
\begin{equation}\label{NKRF.eq}
\frac{\partial\overline{\omega}_t}{\partial t}=-\overline{\rho}_t+\lambda\overline{\omega}_t,
\end{equation}
so that its stationary points are precisely the KE manifolds with given constant $\lambda$, which is usually taken to be in $\{-1,0,1\}$.  
We can determine which constant $\lambda$ to take a priori by noting whether $c_1(M)$ is negative, zero or positive.  It is known by 
work of Cao \cite{Cao} that if $M$ is compact with $c_1(M)\leq 0$ then the normalized K\"ahler--Ricci flow exists for all time and converges to the 
appropriate KE metric.  The case $c_1(M)>0$ (the case of Fano manifolds) is still open in general but should be related to so-called \emph{K-stability}, given the fact
 that K-stability of compact Fano manifolds is equivalent to the existence of a KE metric, as shown in work by Chen--Donaldson--Sun, Tian et al.
\end{remark}

The result we want to review is the following, cf.~\cite{Smo}.

\begin{thm}\label{main.LMCF.thm}
Let $\iota:L\rightarrow M$ be a compact Lagrangian submanifold of a 
K\"ahler manifold $(M,J,\overline{\omega})$.
\begin{itemize}
\item  If $M$ is K\"ahler--Einstein and $\iota_t:L\rightarrow M$ 
satisfies MCF 
with $\iota_0=\iota$ 
then 
$\iota_t:L\rightarrow M$ is Lagrangian for all $t>0$ for which the flow exists.
\item If $\overline{\omega}_t$ satisfies K\"ahler--Ricci flow 
with $\overline{\omega}_0=\overline{\omega}$ and 
$\iota_t:L\rightarrow (M,J,\overline{\omega}_t)$ satisfies MCF  
with $\iota_0=\iota$ 
then  
$\iota_t:L\rightarrow (M,J,\overline{\omega}_t)$ is Lagrangian for all $t>0$ for which the flow exists.
\end{itemize}
\end{thm}

\noindent  The first part of Theorem \ref{main.LMCF.thm} shows that Lagrangian MCF is a well-defined concept.  We should emphasise in the second part that the mean curvature vector $H_t$ at time $t$ is calculated using the ambient metric $\overline{g}_t$ at time $t$ determined by the Ricci flow, so in this sense we have coupled MCF with KRF. In other words, we treat the two equations as a system. Notice however that the coupling is only partial: the ambient flow does not depend on the flow of the submanifold.

Hopefully, our coordinate-free presentation of this result will be a useful complement to the existing literature. It should also serve to emphasize how to group together 
the many terms which in a coordinate-based expression would appear in the formulae, so as to obtain well-defined tensors. Studying the role of these  tensors will be 
a key part of our subsequent improvement on this result.

\subsection{Totally real submanifolds in K\"ahler manifolds}

Let $L$ be an orientable (real) $n$-manifold and let $\iota:L\rightarrow M$ be an immersion. Let us identify $L$ with its image $\iota(L)\subset M$. The relationship $\overline{\omega}(\cdot,\cdot)=\overline{g}(J\cdot,\cdot)$ shows that $\iota$ (or, using the identification, $L$) is Lagrangian if and only if, for all $p\in L$, we have that $J(T_pL)=(T_pL)^\perp$, the normal space. This yields the orthogonal splitting
$$T_pM=T_pL 
\stackrel{\perp}{\oplus} J(T_pL).$$
More generally, we say that $\iota$ (or $L$) is \emph{totally real} if, for all $p\in L$,
$$J(T_pL)\cap T_pL=\{0\}.$$
In this case we still get a splitting of $T_pM$ as above, though it is not necessarily orthogonal.
We can thus write any vector field $Z$ on a totally real $L$ uniquely as $Z=X+JY$ where $X,Y$ are tangent vector fields.  Notice that if we define projections $\pi_L,\pi_J$ by
 $\pi_L(Z)=X$ and $\pi_J(Z)=JY$ then we have the following.
 
 \begin{lem}\label{proj.lem}
 $\pi_L\circ J=J\circ\pi_J$ and $J\circ\pi_L=\pi_J\circ J$.
\end{lem}

\begin{proof} We calculate
$$\pi_L(JZ)=\pi_L(JX-Y)=-Y=J\pi_J(Z)$$
and
$$J\pi_L(Z)=J\pi_L(X+JY)=JX=\pi_J(JX-Y)=\pi_J(JZ).$$
The result follows.
\end{proof}

For Lagrangian MCF the key object to study is 
$\omega=\iota^*\overline{\omega}$, which we want to show stays 
zero along the flow if it is initially zero.  
To achieve this we need some basic objects.

We let
 $g=\iota^*\overline{g}$ and $\nabla$ be the Levi-Civita connection of $g$. 
 We let the second fundamental form of $L$ be given by
$$A(X,Y)=\overline{\nabla}_XY-\nabla_XY$$
for vector fields $X,Y$ on $L$, which takes values in the normal bundle $NL=(TL)^{\perp}$ of $L$ because $\nabla$ is the tangential part of $\overline{\nabla}$ on $L$.  We 
let $H$ be the mean curvature vector of $L$, which is given at $p\in L$ by:
$$H(p)=\overline{\nabla}_{e_i}e_i-\nabla_{e_i}e_i$$
where $\{e_1,\ldots,e_n\}$ is an orthonormal basis of $T_pL$ and we will always sum over repeated indices.  

We now have a general lemma, which shows that we can evaluate $\overline{\omega}$ on $NL$ in terms of $\omega$.  
We let $\pi_{\rm T}$ and $\pi_{\perp}$ denote the tangential and normal projections of a vector field along $L$.

\begin{lem}\label{omega.perp.lem}
Let $Z,W$ be normal vector fields on $L$.  Since $L$ is totally real there exist unique tangent vector fields $X,Y$ on $L$ such that $Z=\pi_{\perp}(JX)$ and $W=\pi_{\perp}(JY)$.   Then 
$$\overline{\omega}(Z,W)=\omega(\pi_{\rm T}(JX),\pi_{\rm T}(JY))-\omega(X,Y).$$
\end{lem}

\begin{proof} This is an elementary calculation:
\begin{align*}
 \overline{\omega}&(\pi_{\perp}(JX),\pi_{\perp}(JY))\\
&=\overline{\omega}(JX-\pi_{\rm T}(JX),JY-\pi_{\rm T}(JY))\\
&=\overline{\omega}(JX,JY)-\overline{g}(J\circ\pi_{\rm T}(JX),JY)+\overline{g}(X,\pi_{\rm T}(JY))+\omega(\pi_{\rm T}(JX),\pi_{\rm T}(JY))\\
&=\omega(X,Y)-\overline{g}(JX,Y)+\overline{g}(X,JY)+\omega(\pi_{\rm T}(JX),\pi_{\rm T}(JY))\\
&=\omega(\pi_{\rm T}(JX),\pi_{\rm T}(JY))-\omega(X,Y),
\end{align*}
 from which the result follows.
\end{proof}

As we have seen in \eq{KRF.eq}, the 2-form $\overline{\rho}$ defines the motion of $\overline{\omega}$ by KRF, so 
it is clearly important to understand how $\overline{\rho}$ restricts to the submanifold $L$.  This is the content of the
following elementary lemma.  Recall that we let $\{e_1,\ldots,e_n\}$ denote an orthonormal basis for $T_pL$ for a 
point $p\in L$.

\begin{lem}\label{rho.lem}
Let $\rho=\iota^*\overline{\rho}$. Then
\begin{align*}
 \rho(X,Y)=\overline{\omega}(\pi_J\overline{R}(X,Y)e_i,e_i).
\end{align*}
\end{lem}

\begin{proof}
Given $p\in L$ and an orthonormal basis $\{\overline{e_1},\ldots,\overline{e}_{2n}\}$ for $T_pM$, we see that $\{J\overline{e_1},\ldots,J\overline{e_{2n}}\}$ is an orthonormal basis for $T_pM$ so
\begin{align*}
 \overline{\rho}(X,Y)&=\overline{\Ric}(JX,Y)\\
&=\overline{\Ric}(Y,JX)\\
&=\overline{g}(\overline{R}(Y,J\overline{e_j})J\overline{e_j},JX)\\
&=\overline{g}(\overline{R}(Y,J\overline{e_j})\overline{e_j},X),
\end{align*}
using the fact that $\overline{\nabla}J=0$ (which is part of the K\"ahler condition).  Thus, using the Bianchi identity we have that:
\begin{align*}
 \overline{\rho}(X,Y)&=-\overline{g}(\overline{R}(Y,X)J\overline{e_j},\overline{e_j})-\overline{g}(\overline{R}(Y,\overline{e_j})X,J\overline{e_j})\\
&=\overline{\omega}(\overline{R}(X,Y)\overline{e_j},\overline{e_j})+\overline{g}(\overline{R}(Y,\overline{e_j})J\overline{e_j},X)\\
&=\overline{\omega}(\overline{R}(X,Y)\overline{e_j},\overline{e_j})-\overline{\Ric}(Y,JX).
\end{align*}
Hence,
\begin{align}\label{overline.rho.eq}
 \overline{\rho}(X,Y)=\frac{1}{2}\overline{\omega}(\overline{R}(X,Y)\overline{e_j},\overline{e_j}).
\end{align}
From this we can deduce that 
\begin{align*}
 \rho(X,Y)&=\frac{1}{2}\overline{\omega}(\overline{R}(X,Y)\overline{e_j},\overline{e_j})\\
&=\frac{1}{2}\overline{g}(\pi_L\circ J\overline{R}(X,Y)e_i,e_i)+\frac{1}{2}\overline{g}(\pi_J\circ J\overline{R}(X,Y)Je_i,Je_i)
\end{align*}
since we can take a trace with respect to a basis of unit vectors for $T_pM$ which is not orthonormal (namely, $\{e_1,\ldots,e_n,Je_1,\ldots, Je_n\}$) by taking appropriate projections.
Hence, by Lemma \ref{proj.lem},
\begin{align*}
\rho(X,Y)&=\frac{1}{2}\overline{g}(J\circ\pi_J\overline{R}(X,Y)e_i,e_i)-\frac{1}{2}\overline{g}(\pi_J\overline{R}(X,Y)e_i,Je_i)\\
&=\overline{\omega}(\pi_J\overline{R}(X,Y)e_i,e_i),
\end{align*}
again using the fact that $\overline{\nabla}J=0$.
\end{proof}

\subsection{The mean curvature vector and the K\"ahler form}\label{kahler_mean_curvature}

We want to compute how $\omega$ varies along MCF, so we need to calculate $\mathcal{L}_{H}\overline{\omega}$
on $L$.  By Cartan's formula and the fact that $\overline{\omega}$ is closed, we see that $\mathcal{L}_{H}\overline{\omega}=\d(H\lrcorner\overline{\omega})$.  Hence,
the first thing we need to calculate is the interior product of $H$ and $\overline{\omega}$.    Again we use an orthonormal basis $\{e_1,\ldots,e_n\}$ for $T_pL$.

\begin{prop}\label{H.hook.omega} For tangent vectors to $L$,
\begin{equation}\label{H.hook.omega.eq}
\overline{\omega}(H,X)=-\d^*\omega(X)-\overline{\omega}\big(e_i,A(e_i,X)\big).
\end{equation}
\end{prop}

\begin{remark}
This result effectively expresses a formula for commuting derivatives with contractions, as can be seen from the proof below.
\end{remark}

\begin{proof}
Let $p\in L$, let $X\in T_pL$ and let $\{e_1,\ldots,e_n\}$ be an 
orthonormal basis for $T_pL$. We have:
\begin{align}\label{cont.eq.1}
\overline{\omega}(\overline{\nabla}_{e_i}e_i,X)
&=\overline{\nabla}_{e_i}\big(\overline{\omega}(e_i,X)\big)-(\overline{\nabla}_{e_i}\overline{\omega})(e_i,X)-\overline{\omega}(e_i,\overline{\nabla}_{e_i}X).
\end{align}
Since $M$ is K\"ahler we know that $\overline{\nabla}\overline{\omega}=0$ so the second term vanishes.  The first term is a derivative of a function on $L$, so we can replace $\overline{\omega}$
by $\omega$ and $\overline{\nabla}$ by $\nabla$ in \eq{cont.eq.1} leading to:
\begin{align}\label{cont.eq.2}
\overline{\omega}(\overline{\nabla}_{e_i}e_i,X)
&=\nabla_{e_i}\big(\omega(e_i,X)\big)-\overline{\omega}(e_i,\overline{\nabla}_{e_i}X).
\end{align}
Now we can use the formula for the codifferential to show that:
\begin{align}
 \d^*\omega(X)&=-\big((e_i\lrcorner\nabla_{e_i})\omega\big)(X)\nonumber\\
&=-(\nabla_{e_i}\omega)(e_i,X)\nonumber\\
&=-\nabla_{e_i}\big(\omega(e_i,X)\big)+\omega(\nabla_{e_i}e_i,X)+\omega(e_i,\nabla_{e_i}X).\label{cont.eq.3}
\end{align}
Hence, we can go back to \eq{cont.eq.2} and substitute in \eq{cont.eq.3}:
\begin{align*}
 \overline{\omega}(\overline{\nabla}_{e_i}e_i,X)&=-\d^*\omega(X)+\omega(\nabla_{e_i}e_i,X)+\omega(e_i,\nabla_{e_i}X)-\overline{\omega}(e_i,\overline{\nabla}_{e_i}X).
\end{align*}
Rearranging we have that
\begin{align*}
 \overline{\omega}(\overline{\nabla}_{e_i}e_i-\nabla_{e_i}e_i,X)&=-\d^*\omega(X)-\overline{\omega}(e_i,\overline{\nabla}_{e_i}X-\nabla_{e_i}X),
\end{align*}
so the result follows.
\end{proof}

\begin{remark}
 Notice that this proposition does not use the totally real condition and that the K\"ahler condition is crucially used to remove a term involving the derivative of $\overline{\omega}$.  
 \end{remark}
 
Let us now define a 1-form $\xi$ on $L$ by:
\begin{equation}\label{old.xi.eq}
\xi(X)=-\overline{\omega}\big(e_i,A(e_i,X)\big),
\end{equation}
which is the second and significant term on the right-hand side of \eq{H.hook.omega.eq}.

\begin{remark}
 Notice that we can rewrite 
\begin{equation}\label{xi.trace.eq}
\xi(X)=-\overline{\omega}(e_i,A(X,e_i))=-\overline{g}(Je_i,\pi_{\perp}\overline{\nabla}_Xe_i)=\tr_L(J\pi_{\perp}\overline{\nabla}_X).
\end{equation}
 This will be significant in later sections.
\end{remark} 

Our next step is to differentiate \eq{H.hook.omega.eq} in Proposition \ref{H.hook.omega}.  
We can deal with the first term on the right-hand side
using the well-known Weitzenb\"ock formula and the fact that $\d\omega=0$ so that $\Delta\omega=(\d\d^*+\d^*\d)\omega=\d\d^*\omega$.

\begin{lem} We have that
\begin{equation}\label{Weitz.eq}
\d\d^*\omega=\nabla^*\nabla\omega+\frac{1}{2}R(e_i,e_j)\omega\cdot e_i^*\cdot e_j^*,
\end{equation}
where $R$ is the Riemann curvature tensor of $g$ extended to act on forms, $e_i^*$ is the dual covector to $e_i$ and $\cdot$ denotes Clifford multiplication; i.e.~
$$\xi\cdot e_i^*=(-1)^k(e_i^*\w\xi+e_i\lrcorner\xi)$$
for a $k$-form $\xi$. 
\end{lem}

  We next 
prove the following important formula for $\d\xi$.

\begin{prop}\label{xi.prop}  For tangent vectors $X,Y$ to $L$,
$$\d\xi(X,Y)=\overline{\omega}\big(\overline{R}(X,Y)e_i,e_i\big)-\omega(R(X,Y)e_i,e_i)-2\overline{\omega}\big(A(X,e_i),A(Y,e_i)\big).$$
\end{prop}

\begin{proof}
By definition,
\begin{align*}
 \d\xi(X,Y)&=X(\xi(Y))-Y(\xi(X))-\xi([X,Y])\\
&=-X\big(\overline{\omega}\big(e_i,A(e_i,Y)\big)\big)+Y\big(\overline{\omega}\big(e_i,A(e_i,X)\big)\big)+\overline{\omega}\big(e_i,A(e_i,[X,Y])\big).
\end{align*}
Now we observe the well-known fact that $A$ is symmetric: 
$$A(Y,X)=\overline{\nabla}_YX-\nabla_YX=\overline{\nabla}_XY+[Y,X]-\nabla_XY-[Y,X]=A(X,Y)$$
since $\overline{\nabla}$ and $\nabla$ are torsion-free.  Thus we have that
\begin{align*}
 \d\xi(X,Y)&=-X\big(\overline{\omega}\big(e_i,A(Y,e_i)\big)\big)+Y\big(\overline{\omega}\big(e_i,A(X,e_i)\big)\big)+\overline{\omega}\big(e_i,A([X,Y],e_i)\big)\\
&=-\overline{\nabla}_X\big(\overline{\omega}\big(e_i,\overline{\nabla}_Ye_i-\nabla_Ye_i\big)\big)+\overline{\nabla}_Y\big(\overline{\omega}\big(e_i,\overline{\nabla}_Xe_i-\nabla_Xe_i\big)\big)\\
&\quad+\overline{\omega}\big(e_i,\overline{\nabla}_{[X,Y]}e_i-\nabla_{[X,Y]}e_i\big)\\
&=-\overline{\omega}\big(\overline{\nabla}_Xe_i,A(Y,e_i)\big)-\overline{\omega}(e_i,\overline{\nabla}_X\overline{\nabla}_Ye_i-\overline{\nabla}_X\nabla_Ye_i)\\
&\quad+\overline{\omega}\big(\overline{\nabla}_Ye_i,A(X,e_i)\big)+\overline{\omega}(e_i,\overline{\nabla}_Y\overline{\nabla}_Xe_i-\overline{\nabla}_Y\nabla_Xe_i)\\
&\quad+\overline{\omega}\big(e_i,\overline{\nabla}_{[X,Y]}e_i-\nabla_{[X,Y]}e_i\big),
\end{align*}
where we have used the fact that $\overline{\nabla}\overline{\omega}=0$ (as $M$ is K\"ahler).
Substituting $\overline{R}(X,Y)=\overline{\nabla}_X\overline{\nabla}_Y-\overline{\nabla}_Y\overline{\nabla}_X-\overline{\nabla}_{[X,Y]}$ and using the formula for the second fundamental form,
we see that
\begin{align*}
 \d\xi(X,Y)&=-\overline{\omega}\big(\nabla_Xe_i+A(X,e_i),A(Y,e_i)\big)-\overline{\omega}\big(e_i,\overline{R}(X,Y)e_i\big)\\
&\quad+\overline{\omega}\big(e_i,\nabla_X\nabla_Ye_i+A(X,\nabla_Ye_i)\big)+\overline{\omega}\big(\nabla_Ye_i+A(Y,e_i),A(X,e_i)\big)\\
&\quad-\overline{\omega}\big(e_i,\nabla_Y\nabla_Xe_i+A(Y,\nabla_Xe_i)\big)-\overline{\omega}\big(e_i,\nabla_{[X,Y]}e_i\big)\\
&=\overline{\omega}\big(\overline{R}(X,Y)e_i,e_i\big)-\overline{\omega}(R(X,Y)e_i,e_i)-2\overline{\omega}\big(A(X,e_i),A(Y,e_i)\big)\\
&\quad-\overline{\omega}\big(\nabla_Xe_i,A(Y,e_i)\big)-\overline{\omega}\big(e_i,A(Y,\nabla_Xe_i)\big)\\
&\quad+\overline{\omega}\big(\nabla_Ye_i,A(X,e_i)\big)+\overline{\omega}\big(e_i,A(X,\nabla_Ye_i)\big).
\end{align*}

We can re-arrange this as
\begin{align*}
\d\xi(X,Y)&-\overline{\omega}\big(\overline{R}(X,Y)e_i,e_i\big)+\omega(R(X,Y)e_i,e_i)+2\overline{\omega}\big(A(X,e_i),A(Y,e_i)\big)\\
&=-\overline{\omega}\big(\nabla_Xe_i,A(Y,e_i)\big)-\overline{\omega}\big(e_i,A(Y,\nabla_Xe_i)\big)\\
&\quad+\overline{\omega}\big(\nabla_Ye_i,A(X,e_i)\big)+\overline{\omega}\big(e_i,A(X,\nabla_Ye_i)\big).
\end{align*}
Notice that the terms on the left-hand side of this equation are tensorial and so are independent of the choice of coordinates we use, and hence the same must be true of the right-hand side.  Therefore, at $p$, we may choose 
geodesic normal 
coordinates, which then means that $$\nabla_Xe_i=\nabla_Ye_i=0\quad\text{at $p$.}$$  This forces all the terms on the right-hand side to vanish (as $A$ is a tensor) and 
we thus deduce the result.
\end{proof}

We can now combine these observations to compute $\mathcal{L}_{H}\overline{\omega}$ on $L$.  

\begin{prop}\label{d.H.hook.omega.prop}
There exists a smooth tensor $C$ on $L$, depending only on $\pi_L\overline{R}$, $R$ and $A$, such that 
\begin{align}\label{eq:d.H.hook.omega}
\iota^*\d(H\lrcorner\overline{\omega})=\rho-\nabla^*\nabla\omega+C\lrcorner\omega.
\end{align}
\end{prop}

\begin{proof}
We first notice that, by Proposition \ref{H.hook.omega},
\begin{align*}
\d(H\lrcorner\overline{\omega})(X,Y)&=-\d\d^*\omega(X,Y)+\d\xi(X,Y).
\end{align*}
Using the Weizenb\"ock formula \eq{Weitz.eq} and Proposition \ref{xi.prop} gives us that
\begin{align*}
-\d\d^*\omega(X,Y)&+\d\xi(X,Y)=-\nabla^*\nabla\omega-\frac{1}{2}R(e_i,e_j)\omega\cdot e_i^*\cdot e_j^*\\
&+\overline{\omega}\big(\overline{R}(X,Y)e_i,e_i\big)-\omega(R(X,Y)e_i,e_i)-2\overline{\omega}\big(A(X,e_i),A(Y,e_i)\big).
\end{align*}
Applying Lemma \ref{omega.perp.lem} we can write
\begin{align*}
\overline{\omega}\big(A(X,e_i),A(Y,e_i)\big)&=\omega\big(\pi_{\rm T}\circ\pi_JA(X,e_i),\pi_{\rm T}\circ\pi_JA(Y,e_i)\big)\\
&\quad-\omega\big(\pi_L\circ JA(X,e_i),\pi_L\circ JA(Y,e_i)\big).
\end{align*}
Applying Lemma \ref{rho.lem} we see that
\begin{align*}
\overline{\omega}\big(\overline{R}(X,Y)e_i,e_i\big)&=\overline{\omega}\big(\pi_J\overline{R}(X,Y)e_i,e_i\big)+\overline{\omega}\big(
\pi_L\overline{R}(X,Y)e_i,e_i\big)\\
&=\rho(X,Y)+\omega(\pi_L\overline{R}(X,Y)e_i,e_i).
\end{align*}
Overall we have that
\begin{align*}
\d(H\lrcorner\overline{\omega})(X,Y)&=\rho(X,Y)-\nabla^*\nabla\omega-\frac{1}{2}R(e_i,e_j)\omega\cdot e_i^*\cdot e_j^*+\omega\big(\pi_L\overline{R}(X,Y)e_i,e_i\big)\\
&\quad-\omega(R(X,Y)e_i,e_i)-2\omega\big(\pi_{\rm T}\circ\pi_JA(X,e_i),\pi_{\rm T}\circ\pi_JA(Y,e_i)\big)\\
&\quad+2\omega\big(\pi_L\circ JA(X,e_i),\pi_L\circ JA(Y,e_i)\big),
\end{align*}
which gives the result.
\end{proof}

\subsection{Lagrangian deformations}\label{ss:lagr_defs}

If $\iota:L\rightarrow M$ is Lagrangian, then the normal bundle $NL=J(TL)$ is isometric to $TL$ and thus $T^*L$, and the map is given by
$$JX\mapsto \overline{\omega}(JX,.)=-g(X,.).$$
We thus can view (normal) deformations of $L$ as 1-forms $\alpha$ and it follows from Weinstein's Lagrangian Neighbourhood Theorem that $\alpha$ defines 
a Lagrangian deformation of $L$ if and only if $\alpha$ is closed.  In other words, if $\mathcal{L}$ is the space of Lagrangian immersions of $L$ in 
$M$ homotopic to $\iota$, up to reparametrisation, then $$T_L\mathcal{L}=\{\alpha\in\Lambda^1(L)\,:\,\d\alpha=0\}.$$ 

A particular normal deformation is given by the mean curvature vector $H$.  Therefore, if we want $H$ to define a Lagrangian deformation of $L$ then we 
need the 1-form $\iota^*(H\lrcorner\overline{\omega})$ to be closed.  We see from Proposition \ref{H.hook.omega} and \eqref{old.xi.eq} that
$$\iota^*(H\lrcorner\overline{\omega})=-\d^*\omega+\xi=\xi$$
as $\omega=0$ because $L$ is Lagrangian.  Moreover, Proposition \ref{d.H.hook.omega.prop} shows that 
$$\d\xi=\rho-\nabla^*\nabla\omega+C\lrcorner\omega=\rho,$$
since $\omega=0$.  Thus $\xi$ is closed if and only if $\rho=0$.  

Now, if $M$ is KE then $\overline{\rho}=\lambda\overline{\omega}$, so $\rho=\iota^*\overline{\rho}=\lambda\omega=0$.  Therefore, for Lagrangians $L$ in 
KE manifolds, we have that $\xi\in T_L\mathcal{L}$, which we write as a lemma, previously given for example in \cite{Oh2}.

\begin{lem} \label{l:Hhookomega}
If $L$ is a Lagrangian submanifold of a K\"ahler--Einstein manifold, then  $H\lrcorner\overline{\omega}=\xi$ is a closed 1-form on $L$ and so 
$H$ defines a Lagrangian deformation of $L$.
\end{lem}

\noindent This key calculation is a compelling reason (which turns out to 
be justified) to believe that MCF preserves Lagrangians in KE manifolds.  It is not however sufficient since it is only 
an infinitesimal deformation calculation.  It also shows that $\xi$ is generally not closed if $M$ is not KE, so 
one should not expect Lagrangians to be preserved by MCF in general K\"ahler manifolds. This justifies the need for something more sophisticated, \textit{i.e.}~the use of coupled flows.

\subsection{Lagrangian MCF in K\"ahler manifolds}\label{ss:existence_lagr_mcf}

Let us suppose that a totally real submanifold $\iota:L\rightarrow M$ evolves via mean curvature flow (MCF) as in \eq{MCF.eq}.
This flow is known to have short-time existence, so we have a one-parameter family of solutions $\iota_t:L\rightarrow M$ with $\iota_0=\iota$ 
and we let $L_t=\iota_t(L)$. 

Let us also suppose that $M$ is K\"ahler--Einstein, so $\overline{\rho}=\lambda\overline{\omega}$ for some constant 
$\lambda$.  We want to show that if $L$ is initially Lagrangian then it remains Lagrangian for all time.  Precisely, we 
show the following, which coincides with the first part of Theorem \ref{main.LMCF.thm}.

\begin{thm}\label{LMCF.KE.thm}
Let $L$ be a compact Lagrangian in a K\"ahler--Einstein manifold.  If $L$ evolves via MCF then 
$L_t$ is Lagrangian for all $t$. 
\end{thm}

\begin{proof}
If we let $g_t=\iota_t^*\overline{g}$ and $\omega_t=\iota_t^*\overline{\omega}$,  
we wish to calculate $\frac{\partial}{\partial t}g_t(\omega_t,\omega_t),$ where we 
consider the metric extended to forms in the natural manner.

We first see that
\begin{align}\label{evolve.eq.1}
\frac{\partial}{\partial t}g_t(\omega_t,\omega_t)=\left(\frac{\partial}{\partial t}g_t\right)(\omega_t,\omega_t)+2g_t\left(\frac{\partial}{\partial t}\omega_t,\omega_t\right).
\end{align}

Now,
\begin{align}
\frac{\partial}{\partial t}\omega_t&=\frac{\partial}{\partial t}\iota_t^*\overline{\omega}=\iota_t^*\mathcal{L}_{H_t}\overline{\omega}=\iota_t^*\d(H_t\lrcorner\overline{\omega})+\iota_t^*(H_t\lrcorner\d\overline{\omega}),\label{evolve.eq.2}
\end{align}
using Cartan's formula.  Since $\d\overline{\omega}=0$ the second term vanishes in \eq{evolve.eq.2}.  
 For small $t$, $\iota_t$ is a totally real immersion since the totally real condition is an open one, 
so we can apply Proposition \ref{d.H.hook.omega.prop} and deduce
that
\begin{align}\label{evolve.eq.3}
\iota_t^*\d(H_t\lrcorner\overline{\omega})=\rho_t-\nabla^*_t\nabla_t\omega_t+C_t\lrcorner\omega_t,
\end{align}
where $\rho_t=\iota_t^*\overline{\rho}$, $\nabla_t$ is the Levi-Civita connection of $g_t$ and
$C_t$ is a smooth tensor only depending on the second fundamental form and the Riemann curvature tensor on $L_t$.

Now $\overline{\rho}=\lambda\overline{\omega}$ by the K\"ahler--Einstein condition, hence $\rho_t=\iota_t^*\overline{\rho}=\lambda \iota_t^*\overline{\omega}=\lambda\omega_t$.
Plugging \eq{evolve.eq.3} in \eq{evolve.eq.2}  gives us that
\begin{align}\label{evolve.eq.5a}
 \frac{\partial}{\partial t}\omega_t=\rho_t-\nabla^*_t\nabla_t\omega_t+C_t\lrcorner\omega_t=-\nabla^*_t\nabla_t\omega_t+C'_t\lrcorner\omega_t.
\end{align}
for some smooth tensor $C'_t$.
Therefore, we see that
\begin{align}
g_t\left(\frac{\partial}{\partial t}\omega_t,\omega_t\right)&=g_t(-\nabla^*_t\nabla_t\omega_t+C'_t\lrcorner\omega_t,\omega_t)\nonumber\\
&=-\frac{1}{2}\nabla_t^*\nabla_t\big(g_t(\omega_t,\omega_t)\big)-g_t(\nabla_t\omega_t,\nabla_t\omega_t)+g_t(C'_t\lrcorner\omega_t,\omega_t).\label{evolve.eq.5}
\end{align}

Inserting \eq{evolve.eq.5} in \eq{evolve.eq.1} allows us to deduce that
\begin{align}
\frac{\partial}{\partial t}g_t(\omega_t,\omega_t)&=-\nabla_t^*\nabla_t\big(g_t(\omega_t,\omega_t)\big)-2g_t(\nabla_t\omega_t,\nabla_t\omega_t)+2g_t(C'_t\lrcorner\omega_t,\omega_t)\nonumber\\
&\qquad+\left(\frac{\partial}{\partial t}g_t\right)(\omega_t,\omega_t).\label{evolve.eq.6}
\end{align}
We see that $g_t(\nabla_t\omega_t,\nabla_t\omega_t)\geq 0$.  

Choose a finite time $T>0$ such that $L_t$ is defined for all $t\in [0,T]$ and let $f=g_t(\omega_t,\omega_t)$, which is non-negative and vanishes at $t=0$.  As $L$ is compact,  \eqref{evolve.eq.6} implies that $f$ satisfies the parabolic inequality
\begin{equation}\label{eq:parabolic_inequality}
\frac{\partial}{\partial t}f\leq -\nabla_t^*\nabla_t f+Bf
\end{equation}
for some constant $B$.  Applying the maximum principle to \eqref{eq:parabolic_inequality} gives us that $f\equiv 0$, and thus $L_t$ is Lagrangian for all $t$.
\end{proof}

\noindent We see from the proof that without the KE assumption there is no reason to suppose that Lagrangians are preserved by MCF in a K\"ahler manifold, as we 
knew from the infinitesimal deformation calculation from the previous subsection.  

However, we now show that if $M$ is any K\"ahler manifold evolving under K\"ahler--Ricci flow then Lagrangians are still preserved under MCF.  This of course contains our previous result as a special case because 
KE manifolds are solitons for KRF, where the K\"ahler form simply evolves by dilations, and hence the space of Lagrangian submanifolds stays the same for all time.
  The idea is that KRF exactly ``cancels out'' $\d(H\lrcorner\overline{\omega})$, because $\xi=H\lrcorner\overline{\omega}$ is not closed, thus ensuring that
  $H$ becomes tangent to the space of Lagrangian immersions (which now varies with $t$ as the K\"ahler structure is varying).

Let us continue to suppose that $L$ evolves by MCF but we also
suppose that simultaneously the K\"ahler structure $(\overline{g},J,\overline{\omega})$ on $M$ is evolving via KRF as in \eq{KRF.eq}.  Note that 
we fix the complex structure $J$ and we let $\overline{g}_t$ be the metric such that $\overline{g}_t(JX,Y)=\overline{\omega}_t(X,Y)$ for all $X,Y$.

Since the complex structure $J$ is fixed, the notion of totally real in $M$ is independent of $t$.  Moreover, the condition for an immersion to be totally real is an open one and both Ricci flow and MCF have short-time existence, so for short time we know that $\iota_t:L\rightarrow M$ exists and remains totally real.   The following result coincides with the second part of Theorem \ref{main.LMCF.thm}.

\begin{thm}\label{LMCF.KRF.thm}
Let $M$ be a K\"ahler manifold evolving under K\"ahler--Ricci flow and let $L$ be a compact submanifold of $M$ evolving 
simultaneously under mean curvature flow, which is Lagrangian for the initial K\"ahler structure.  Then $L_t$ is Lagrangian for the 
K\"ahler structure $\overline{\omega}_t$ for all $t$. 
\end{thm}

\begin{proof}
Using the same notation as in the proof of Theorem \ref{LMCF.KE.thm} we see that
\begin{align}
\frac{\partial}{\partial t}\omega_t&=\frac{\partial}{\partial t}\iota_t^*\overline{\omega}_t
=\iota_t^*\mathcal{L}_{H_t}\overline{\omega}_t+\iota_t^*\frac{\partial}{\partial t}\overline{\omega}_t
=\iota_t^*\d(H_t\lrcorner\overline{\omega}_t)-\iota_t^*\overline{\rho}_t,\label{evolve.eq.2b}
\end{align}
using Cartan's formula, $\d\overline{\omega}_t=0$ (as they are K\"ahler forms) and the fact that $\overline{\omega}_t$ evolves by K\"ahler--Ricci flow \eq{KRF.eq}.
As in the proof of Theorem \ref{LMCF.KE.thm} we have that \eqref{evolve.eq.3} holds, but now with $\rho_t=\iota_t^*\overline{\rho}_t$.  Substituting \eqref{evolve.eq.3} in \eqref{evolve.eq.2b} gives
\begin{align*}
 \frac{\partial}{\partial t}\omega_t=\rho_t-\nabla^*_t\nabla_t\omega_t+C_t\lrcorner\omega_t-\rho_t=-\nabla^*_t\nabla_t\omega_t+C_t\lrcorner\omega_t,
\end{align*}
which is the same form as \eqref{evolve.eq.5a}.  The proof now proceeds just as for Theorem \ref{LMCF.KE.thm}.
\end{proof}

We observe a minor modification of the previous result. The proof is similar.

\begin{cor}\label{cor:normalized}
Let $M$ be a K\"ahler manifold evolving under the normalized K\"ahler--Ricci flow as in \eq{NKRF.eq} and let $L$ be a compact Lagrangian in $M$ 
with respect to the initial K\"ahler structure.  If $L$ evolves simultaneously by mean curvature flow then $L_t$ is Lagrangian with respect to $\overline{\omega}_t$ 
for all $t$.
\end{cor}

\paragraph{Concluding remarks.}Before closing this section let us pause to reflect upon two issues.
First, notice the basic strategy underlying these proofs. The goal is to show that, under certain conditions, the Lagrangian condition is preserved. What could go wrong? Given that Lagrangians form a closed subset of the space of totally real submanifolds, clearly the one thing we need to rule out is that the initial submanifold ``degenerates'', becoming totally real. The strategy is thus to extend the study of the volume functional and MCF from Lagrangians to totally reals. Compare however the formulae for Lagrangians, appearing in Section \ref{ss:lagr_defs}, with those for general totally real submanifolds, appearing in (\ref{H.hook.omega.eq}) and (\ref{eq:d.H.hook.omega}). The latter are clearly more cumbersome, leading to additional complications in the final steps of the proofs. Specifically, 
\begin{itemize}
\item these extra terms lead to the parabolic inequality  (\ref{eq:parabolic_inequality}), forcing us to use the maximum principle;
 \item the final result concerns only Lagrangians, not totally reals.
\end{itemize}
The calculations also rely crucially on the K\"ahler condition, without which we could not say anything even about Lagrangians.

Secondly, Lemma \ref{l:Hhookomega} indicates that, in the appropriate context, the mean curvature vector $H$  has a double geometric description: both as the negative gradient of the volume functional and as a primitive of the Ricci form. Furthermore, it satisfies a certain differential equation: the 1-form $\iota^*(H\lrcorner\bar{\omega})$ is closed. 

In the following sections we will investigate both issues in depth. Specifically, we will initiate a study of totally real submanifolds which will indicate that the standard Riemannian volume is not a particularly natural quantity in this context. Replacing its role in the above proofs with its more natural analogues will lead to uniformly simple formulae governing the flow both of Lagrangians and of totally reals, yielding stronger results and simpler proofs. Ultimately, it will be the second geometric description of the flow  which will take the lead.

\section{Geometry of totally real submanifolds}\label{s:totally_real}

Let $(M,J)$ be a $2n$-manifold endowed with an almost complex structure. Given $p\in M$, recall that an $n$-plane $\pi$ in $T_pM$ is \textit{totally real} if $J(\pi)\cap\pi=\{0\}$, 
\textit{i.e.} if $T_pM$ is the complexification of $\pi$. We denote by $\text{TR}^+_p$ the Grassmannian of oriented totally real $n$-planes in $T_pM$. 
The union of these spaces defines a fibre bundle $\text{TR}^+$ over $M$, whose fibre is $\GL(n,\C)/\GL^+(n,\R)$. 

Now let $\iota:L\rightarrow M$ be an immersion of an $n$-dimensional oriented manifold $L$ in $M$. As before, we will often identify $L$ with $\iota$ and say that $\iota$ (or $L$) is \textit{totally real} if, for each $p\in L$, $T_pL$ is totally real in $T_pM$. This determines a decomposition
$$T_pM=T_pL\oplus J(T_pL).$$
We let $\pi_L,\pi_J$ denote the projections from $T_pM$ onto $T_pL$ and $J(T_pL)$ respectively as before.  These are clearly the more natural projections in the context of
totally real geometry than the usual tangential and normal projections, because they do not need the additional structure of a metric. 

We let $\TRL$ denote the space of totally real immersions of $L$ which are homotopic to a given $\iota$, modulo reparametrization. 
Notice that the totally 
real condition is open in the Grassmannian of all $n$-planes, so it is a ``soft'' condition. In particular, the space $\TRL$ is 
infinite-dimensional. 

Given $L\in \TRL$, we describe its tangent space $T_L\TRL$ as follows. Infinitesimal deformations of a given immersion $L$ are 
determined by the space of all sections of the bundle $TM$ over $L$.  At the infinitesimal level, quotienting immersions by reparametrization amounts  to taking 
the quotient of all sections by those which are tangent to $L$. Thus $T_L\TRL$ can be identified with sections of the bundle $TM/TL\simeq J(TL)\simeq TL$ over 
$L$, \textit{i.e.} $T_L\TRL\simeq \Lambda^0(TL)$. The key point here is that the totally real condition provides not only a canonical space in $TM$ which is 
transverse to $TL$, but also a canonical isomorphism with $TL$. In other words, the (extrinsic) ``normal'' bundle (defined via quotients) is canonically isomorphic to the (intrinsic) tangent bundle.

In the next section we show that the totally real condition is closely related to the geometry of the canonical bundle $K_M$ of $M$. This fact determines a ``natural'' (we call it \textit{canonical}) 
geometry of totally real submanifolds. This construction, which was motivated in part by work in \cite{Bor}, requires some additional structure: a metric $h$ and a unitary connection $\tnabla$ on $K_M$.

\subsection{Canonical data for totally real submanifolds}\label{ss:canonical_data}

We can characterize totally real planes in $T_pM$ as follows: an $n$-plane $\pi$ in $T_pM$ is totally real if and only if $\alpha|_{\pi}\neq 0$ 
for all (equivalently, for any) $\alpha\in K_M(p)\setminus\{0\}$.  
This characterization  clearly demonstrates the importance of the canonical bundle $K_M$ for totally real geometry.  
Notice that $n$-planes $\pi$ in $T_pM$ which are not totally real (\textit{i.e.}~those which 
satisfy $\alpha|_{\pi}=0$ for some $\alpha\in K_M(p)\setminus\{0\}$) contain a complex line: a pair $\{X,JX\}$ for some $X\in T_pM\setminus\{0\}$.  We 
call these $n$-planes \emph{partially complex}.  We also call $n$-dimensional submanifolds partially complex if each of their tangent spaces are partially complex.

Let $\pi$ be an oriented totally real $n$-plane in $T_pM$ and let $v_1,\ldots,v_n$ be a positively oriented basis. We may then define $v_j^*\in T_p^*M\otimes\C$ by
$$v_j^*(v_k)=\delta_{jk}\quad\text{and}\quad v_j^*(Jv_k)=i\delta_{jk}.$$
This allows us to define a non-zero form $v_1^*\w\ldots\w v_n^*\in K_M(p)$.

However, the form we have constructed depends on the choice of basis $v_1,\ldots,v_n$. 
We can fix this by assuming that we have a Hermitian metric $h$ on the canonical bundle $K_M$ of $M$. We then define
$$\Omega_J[\pi]=\frac{v_1^*\w\ldots\w v_n^*}{|v_1^*\w\ldots\w v_n^*|_h}\in K_M(p).$$
This form has unit norm and is 
independent of the choice of basis: if we choose another basis $w_1,\ldots,w_n$, which is equivalent to choosing $A\in\GL^+(n,\R)$ such that $w=Av$ (in matrix notation), then $w_1^*\w\ldots\w w_n^*=\det(A^{-1})v_1^*\w\ldots\w v_n^*$ so 
$$\frac{w_1^*\w\ldots\w w_n^*}{|w_1^*\w\ldots\w w_n^*|_h}=\frac{v_1^*\w\ldots\w v_n^*}{|v_1^*\w\ldots\w v_n^*|_h}.$$
We have thus defined a map between bundles $\Omega_J:\text{TR}^+\rightarrow K_M$ covering the identity map on $M$ (in fact, $\Omega_J$ maps into the unit circle bundle in 
$K_M$). 

\begin{definition}
The J-volume form on $\pi\in \text{TR}^+$ is the real-valued $n$-form $\vol_J:=\Omega_J[\pi]|_{\pi}$ obtained by restricting 
the form $\Omega_J[\pi]$ to $\pi$.
\end{definition}

Now let $\iota:L\rightarrow M$ be an $n$-dimensional totally real immersion. We can then obtain global versions of the above constructions as follows. 

\paragraph{Canonical bundle over {\boldmath $L$}.}Let $K_M[\iota]$ denote the pull-back of $K_M$ over $L$. This defines a complex line bundle over $L$ which depends on $\iota$. Specifically, the fibre over $p\in L$ is the fibre of $K_M$ over $\iota(p)\in M$.

Observe that any complex-valued $n$-form $\alpha$ on $T_pL$ defines a unique $n$-form $\widetilde{\alpha}$ 
on $T_{\iota(p)}M$ by identifying $T_pL$ with its image via $\iota_*$ and by setting, \textit{e.g.}, 
$$\widetilde{\alpha}[\iota(p)](J\iota_*(v_1),\dots,J\iota_*(v_n)):=i^n\alpha[p](v_1,\dots,v_n).$$ 
The totally real condition implies that this is an isomorphism: the bundle $K_M[\iota]$ is canonically isomorphic, via $\iota_*$, with the ($\iota$-independent) bundle $\Lambda^n(L,\C):=\Lambda^n(L,\R)\otimes\C$ of complex-valued $n$-forms on $L$. 
 
\paragraph{Canonical section.}Now assume that $L$ is oriented. Then $\Lambda^n(L,\R)$ is trivial, so $K_M[\iota]$ also is. We can build a global section of $K_M[\iota]$ using our previous linear-algebraic construction: $p\mapsto\Omega_J[\iota](p):=\Omega_J[\iota_*(T_pL)]$. We call $\Omega_J[\iota]$ the \textit{canonical section} of $K_M[\iota]$. If we restrict the form $\Omega_J[\iota]$ to $\iota_*(T_pL)$ we obtain a \emph{real-valued} positive $n$-form on $\iota(L)$, thus a volume form $\vol_J[\iota]:=\iota^*(\Omega_J[\iota])$ on $L$: we call it the \textit{$J$-volume form} of $L$, defined by $\iota$.  

When $L$ is compact we obtain a ``canonical volume'' $\int_L \vol_J[\iota]$, for $\iota\in\mathcal{P}$.  
If $\varphi$ is an orientation-preserving diffeomorphism of $L$ then $\vol_J[\iota\circ\varphi]=\varphi^*(\vol_J[\iota])$, just as for the standard volume form, thus 
$$\int_L\vol_J[\iota\circ\varphi]=\int_L\varphi^*\vol_J[\iota]=\int_L\vol_J[\iota].$$  Hence the canonical volume descends to define the \textit{$J$-volume functional}

\begin{equation*}
 \Vol_J:\TRL\rightarrow \R,\ \ L\mapsto\int_L\vol_J.
\end{equation*}

\paragraph{Maslov 1-form.}Now we assume we are given a unitary connection $\tnabla$ on $(K_M,h)$. Such connections always exist. The canonical section then induces a ``connection $1$-form'' $A[\iota]\in \Lambda^1(L,\C)$ defined by the identity
\begin{equation*}
 \tnabla\Omega_J[\iota]=A[\iota]\otimes\Omega_J[\iota],
\end{equation*}
where we are using the pull-back connection on $K_L[\iota]$. Notice that 
\begin{align*}
A[\iota](X)\cdot h(\Omega_J[\iota],\Omega_J[\iota])&=h(\tnabla_X\Omega_J[\iota],\Omega_J[\iota])\\
&=-h(\Omega_J[\iota],\tnabla_X\Omega_J[\iota])\\
&=-\overline{A[\iota]}(X)\cdot h(\Omega_J[\iota],\Omega_J[\iota]),
\end{align*}
where we are using the pull-back metric and the fact that $h(\Omega_J[\iota],\Omega_J[\iota])\equiv 1$. This calculation shows that $A$ actually takes values in $\Imm(\C)$: we shall write $A[\iota]=i\xi_J[\iota]$, calling $\xi_J[\iota]$ the \textit{Maslov $1$-form} on $L$, defined by $\iota$.  One sees from the definition that if $\varphi\in\Diff(L)$, then 
$\varphi^*\xi_J[\iota]=\xi_J[\iota\circ\varphi]$.

\begin{remark}
Of course, one must take into account the fact that this Maslov $1$-form depends on the choice of connection. It really becomes ``canonical'' only in situations 
where the connection itself is canonical. We will see below that this happens, for example, in the context of K\"ahler and almost K\"ahler manifolds.
\end{remark}

\textit{Notation.} We will often simplify notation by dropping the reference to the  immersion used. 
Since this is standard in other contexts, \textit{e.g.} when discussing the Riemannian volume, we expect it will not create any confusion.

\section{Canonical geometry in the Hermitian context}\label{s:hermitian}

Let us now assume that $(M,J)$ is almost Hermitian, $\textit{i.e.}$ we choose a Riemannian metric $\overline{g}$ on $M$ which is compatible with $J$, so 
$J$ is an isometry defining a Hermitian metric $h$ on $M$. We also choose a unitary connection $\tnabla$ on $M$. 
Let $L$ be an oriented totally real submanifold of $(M,J)$.

The structures on $M$ induce structures $h$, $\tnabla$ on $K_M$, which we can use to define the $J$-volume form 
and the Maslov 1-form on $L$. In this context we can use this data to define two natural flows of totally real submanifolds. The goal of this section is to introduce these flows and to compare them to MCF.

Notice that, in contrast to the previous section where we were given only a complex structure, in this section we will be able to also discuss Lagrangian submanifolds, defined via the induced positive $(1,1)$-form $\overline{\omega}$. 

Let $\widetilde{R}$ denote the curvature of $\tnabla$ and for $X,Y\in T_pM$ let 
$$\widetilde{P}(X,Y)=\overline{\omega}(\widetilde{R}(X,Y)\overline{e}_j,\overline{e}_j),$$
where $\overline{e}_1,\ldots,\overline{e}_{2n}$ is an orthonormal basis for $T_pM$.  This defines a closed 2-form $\widetilde{P}$ on $M$. 

According to Chern--Weil theory, $\frac{i}{2\pi}\tr_M\widetilde{R}$ represents the first Chern class of $(M,J)$, where we take the complex trace (\textit{i.e.}~with respect to the Hermitian metric) of the endomorphism part of $\widetilde{R}$. Recall that for a skew-Hermitian endomorphism $E$ we can compare real and complex traces by $\tr^{\R}(E\circ J)=2i\tr^{\C}E$. We deduce the following.

\begin{lem} The 2-form $\widetilde{P}$ satisfies
$\left[\frac{1}{2}\widetilde{P}\right]=2\pi c_1(M)$. 
\end{lem}

Notice that, because the Bianchi identity does not necessarily hold for $\widetilde{R}$, in general $\widetilde{P}$ is not twice the ``Ricci form''
$$\widetilde{\rho}(X,Y)=\overline{g}(\widetilde{R}(JX,\overline{e}_j)\overline{e}_j,Y)$$
and that, as $\tnabla$ is not necessarily the Levi-Civita connection $\overline{\nabla}$, the standard Ricci form $\overline{\rho}$ 
usually does not represent the first Chern class.

\subsection{{\boldmath $J$}-volume versus the Riemannian volume}

In the almost Hermitian context, given an immersion $\iota$, we can define the usual Riemannian volume form $\vol_g$ using the induced metric $g$. It is useful to compare this with the $J$-volume form.

Let us identify $T_pL$ with its image plane in $T_pM$ using $\iota_*$. Given a positively oriented basis $v_1,\dots,v_n$ of $T_pL$, recall that 
$$\vol_{g|p}:=\frac{v_1^*\wedge\dots\wedge v_n^*}{|v_1^*\wedge\dots\wedge v_n^*|_g},$$
where here $v_i^*$ denotes the standard dual basis of $T_p^*L$. Comparing this to $\vol_{J|p}:=\Omega_J[T_pL]_{|T_pL}$, we see that, up to the canonical identifications with $K_M$ discussed above, the two forms differ only by the choice of metric used in the normalization. These definitions imply that 
$$\vol_g(v_1,\dots,v_n)=\sqrt{\det g(v_i,v_j)}\quad\text{and}\quad \vol_J(v_1,\dots,v_n)=\sqrt{{\det}_{\C} h(v_i,v_j)}.$$
We can thus write $\vol_J=\rho_J\vol_g$, 
for $\rho_J$ determined by 
\begin{equation*}\rho_J:\text{TR}^+\rightarrow \R,\ \ \rho_J(\pi):=\vol_J(e_1,\dots,e_n)=\sqrt{{\det}_{\C}h_{ij}}
\end{equation*}
where  $e_1,\ldots,e_n$ is a positive orthonormal basis of $\pi$ and $h_{ij}=h(e_i,e_j)$. Notice that $\rho_J(\pi)$ is well-defined because it is independent of the orthonormal basis chosen for $\pi$.  
Analogously, 
\begin{equation}\label{h.mod.eq}
|e_1^*\w\ldots\w e_n^*|_h=({\det}_{\C}h_{ij})^{-\frac{1}{2}}.
\end{equation}

Since $h=\overline{g}-i\overline{\omega}$, if we set $\omega_{ij}=\bar{\omega}(e_i,e_j)$ we have that
\begin{align*}
{\det}_{\C}h_{ij}&=\sqrt{\det\left(\begin{array}{cc}\delta_{ij} &\omega_{ij}\\ -\omega_{ij} & \delta_{ij}\end{array}\right)}.
\end{align*}
Since $\omega_{ij}=\overline{g}(Je_i,e_j)$ 
and $-\omega_{ij}=\overline{g}(e_i,Je_j)$, we deduce $\det_{\C}h_{ij}=\sqrt{\det(\overline{g}_{ab})}$ where $\overline{g}_{ab}$ is the 
matrix of $\overline{g}$ with respect to the basis $\{e_1,\ldots,e_n,Je_1,\ldots,Je_n\}$.  Therefore
$${\det}_{\C}h_{ij}=\vol_{\overline{g}}(e_1,\ldots,e_n,Je_1,\ldots,Je_n).$$

We thus have a second expression for $\rho_J$: 
\begin{equation}\label{eq:rhoJ}
\rho_J(\pi)=\sqrt{\vol_{\overline{g}} (e_1,\dots,e_n,Je_1,\dots,Je_n)}.
\end{equation}
Hence we see that $\rho_J(\pi)\leq 1$ with equality if and only if $\pi$ is Lagrangian.

We can set $\rho_J(\pi)=0$ when $\pi$ is partially complex and extend the map $\Omega$ to all $n$-planes, just setting $\Omega_J[\pi]=0$ if $\pi$ is partially complex. 
This is particularly reasonable in this almost Hermitian setting, where there is a natural topology on the Grassmannian of $n$-planes: this choice of extension of $\Omega$ would be justified by the fact that it is the unique one which preserves the continuity of $\Omega$.

Applying these observations to submanifolds, we deduce that the $J$-volume functional provides a lower bound for the standard volume.

\begin{lem}\label{Lag.Jvol.lem}
For any compact oriented $n$-dimensional submanifold $L$ in an almost Hermitian
manifold $(M,J,\overline{g})$, we have $\Vol_J(L)\leq \Vol_g(L)$ with equality if and only if $L$ is Lagrangian. In particular the values of $\Vol_J$ and $\Vol_g$ and of their first derivatives coincide on Lagrangians.
\end{lem}

\begin{proof}
The first statement follows from (\ref{eq:rhoJ}). To prove the second, let $L_t$ be a 1-parameter family of totally real submanifolds such that $L_0$ is Lagrangian. Set $f(t):=\Vol_J(L_t)$ and $g(t):=\Vol_g(L_t)$. Then $f\leq g$ so $g-f\geq 0$. Equality holds when $t=0$: this is a minimum point, so it is necessarily critical. It follows that $f'(0)=g'(0)$. The result follows.
\end{proof}

\subsection{Formulae for the Maslov 1-form}

 For $X\in T_pL$ we can define an endomorphism $J\pi_J\tnabla_X:T_pL\rightarrow T_pL$, which depends 
linearly on $X$. The fact that it really is an endomorphism depends on the following calculation:
$$J\pi_J\tnabla_X(fY)=fJ\pi_J\tnabla_XY+J\pi_JX(f)Y=fJ\pi_J\tnabla_XY$$
since $\pi_J(Y)=0$.

\begin{prop}\label{xiJconn.prop}
The Maslov 1-form $\xi_J$ is the trace of the endomorphism $J\pi_J\tnabla$, \textit{i.e.}~for all $X\in T_pL$,
$$\xi_J(X)=\tr_L(J\pi_J\tnabla_X).$$
\end{prop}

\begin{proof}
Let $e_1,\ldots,e_n$ be a positively oriented orthonormal basis of $T_pL$. Recall that we defined the canonical section $\Omega_J$ of $K_M[\iota]$ in terms of the corresponding complexified dual forms in $\Lambda^{1,0}_pM$. We can alternatively express it in terms of the standard dual basis corresponding to the basis $e_1,\ldots,e_n, Je_1,\dots,Je_n$ of $T_pM$:
\begin{equation}\label{OmegaJ.alt.eq}
 \Omega_J(p)=\frac{(e_1^*+i(Je_1)^*)\w\ldots\w (e_n^*+i(Je_n)^*)}{|(e_1^*+i(Je_1)^*)\w\ldots\w (e_n^*+i(Je_n)^*)|_h}.
\end{equation}
For our purposes it is simpler to perform computations using the dual bundle $K_M^*[\iota]$. Set
\begin{equation}\label{nuL.dfn.eq}
\nu:=(e_1-iJe_1)\w\ldots\w (e_n-iJe_n),
\end{equation}
so that $\sigma_J:=\nu/|\nu|$ is a unit section of $K_M^*[\iota]$. 
Observe that
$$\tnabla_X\nu=\sum_{j=1}^n(e_1-iJe_1)\w\ldots\w(\tnabla_Xe_j-iJ\tnabla_Xe_j)\w\ldots\w(e_n-iJe_n),$$
using the fact that $\tnabla_XJ=0$.  Now,
\begin{align*}
\tnabla_Xe_j-iJ\tnabla_Xe_j&=\pi_L\tnabla_Xe_j-iJ\pi_L\tnabla_Xe_j+\pi_J\tnabla_Xe_j-iJ\pi_J\tnabla_Xe_j\\
&=\overline{g}(\pi_L\tnabla_Xe_j,e_k)(e_k-iJe_k)+\overline{g}(\pi_J\tnabla_Xe_j,Je_k)(Je_k+ie_k)\\
&=\big(\overline{g}(\pi_L\tnabla_Xe_j,e_k)-i\overline{g}(J\pi_J\tnabla_Xe_j,e_k)\big)(e_k-iJe_k).
\end{align*}
Hence,
\begin{equation}\label{nuL.eq}
\tnabla_X\nu=\big(\overline{g}(\pi_L\tnabla_Xe_j,e_j)-i\overline{g}(J\pi_J\tnabla_Xe_j,e_j)\big)\nu.
\end{equation}
Furthermore,
\begin{align}
\tnabla_X |\nu|_h^{-1}&=-\frac{|\nu|_h^{-3}}{2}\tnabla_X h(\nu,\nu)\nonumber\\
&=-|\nu|_h^{-3}\Ree\big(h(\tnabla_X\nu,\nu)\big)\nonumber\\
&=-|\nu|_h^{-3}\overline{g}(\pi_L\tnabla_Xe_j,e_j)h(\nu,\nu)\nonumber\\
&=-\overline{g}(\pi_L\tnabla_Xe_j,e_j)|\nu|_h^{-1}.\label{rhoL.eq}
\end{align}
It follows that
\begin{align*}
\tnabla_X\sigma_J
&=\tnabla_X|\nu|_h^{-1} \nu+|\nu|_h^{-1}\tnabla_X\nu\\
&=-\overline{g}(\pi_L\tnabla_Xe_j,e_j)\sigma_J+\big(\overline{g}(\pi_L\tnabla_Xe_j,e_j)-i\overline{g}(J\pi_J\tnabla_Xe_j,e_j)\big)\sigma_J\\
&=-i\overline{g}(J\pi_J\tnabla_Xe_j,e_j)\sigma_J.
\end{align*}
We conclude that
\begin{equation}\label{OmegaL.eq}
\tnabla_X\sigma_J=-i\tr_L(J\pi_J\tnabla_X)\sigma_J.
\end{equation}
This implies that the connection 1-form of the dual section $\Omega_J$ has the opposite sign, proving the claim.
\end{proof}

It follows from the definition that $i\,\d\xi_J$ is the curvature of the complex Hermitian connection on $K_M[\iota]$, thus $-\d\xi_J$ represents $2\pi\, c_1(K_M[\iota])$. The next proposition makes this more explicit, providing one of the key formulae for later results.

\begin{prop}\label{xiJprop}
  For all $X,Y\in T_pL$,
$$\d\xi_J(X,Y)=\tr_L\big(J\pi_J\widetilde{R}(X,Y)\big)=\frac{1}{2}\widetilde{P}(X,Y).$$
\end{prop}

\begin{proof}
By definition of the exterior derivative, 
$$\d\xi_J(X,Y)=X\cdot\tr_L(J\pi_J\tnabla_Y)-Y\cdot\tr_L(J\pi_J\tnabla_X)-\tr_L(J\pi_J\tnabla_{[X,Y]}).$$

If we let $v_1,\ldots,v_n$ denote a basis for $T_pL$, we can write $J\pi_J\tnabla_Y$ as a matrix $E$ with respect to this basis. Letting $v_1^*,\ldots,v_n^*$ denote the natural dual basis for $T_p^*L$, we have
$$\tr_L(J\pi_J\tnabla_Y)=v_i^*(Ev_i),$$
using summation convention.  As we observed, the unitary connection $\tnabla$ on $M$ induces a connection on $L$ given by $\pi_L\tnabla$.
We can use the connection $\pi_L\tnabla$ on $L$ to extend the basis $v_i$ locally by parallel transport to compute
$$X\cdot\tr_L(J\pi_J\tnabla_Y)=\pi_L\tnabla_X\big(v_i^*(Ev_i)\big)=v_i^*\big((\pi_L\tnabla_XE)v_i),$$
where we extend the connection $\pi_L\tnabla$ to endomorphisms.  

Using the Leibniz rule, we see that for any $Z\in T_pL$, 
\begin{align*}
\big(\pi_L\tnabla_X(J\pi_J&\tnabla_Y)\big)Z\\
&=\pi_L\tnabla_X(J\pi_J\tnabla_YZ)-J\pi_J\tnabla_Y(\pi_L\tnabla_XZ)\\
&=\pi_LJ\tnabla_X(\tnabla_YZ-\pi_L\tnabla_YZ)-J\pi_J\tnabla_Y(\pi_L\tnabla_XZ)\\
&=J\pi_J\tnabla_X\tnabla_YZ-J\pi_J\tnabla_X(\pi_L\tnabla_YZ)-J\pi_J\tnabla_Y(\pi_L\tnabla_XZ).
\end{align*}
The last two terms are symmetric in $X$ and $Y$ whereas $\d\xi_J$ is  skew in $X,Y$, so
\begin{align*}
\d\xi_J(X,Y)&=\tr_L(J\pi_J\tnabla_X\tnabla_Y-J\pi_J\tnabla_Y\tnabla_X-J\pi_J\tnabla_{[X,Y]}),
\end{align*}
from which the first part of the result follows.

We now notice that if $e_1,\ldots,e_n$ is an orthonormal basis for $T_pL$ we can extend it to a basis of $T_pM$ consisting of unit vectors by using $Je_1,\ldots,Je_n$.  Then we see that
\begin{align*}
\widetilde{P}(X,Y)&=\overline{\omega}(\widetilde{R}(X,Y)\overline{e}_j,\overline{e}_j)=\overline{g}(J\widetilde{R}(X,Y)\overline{e}_j,\overline{e}_j)=\tr_M(J\widetilde{R}(X,Y))\\
&=e_i^*(J\widetilde{R}(X,Y)e_i)+(Je_i)^*(J\widetilde{R}(X,Y)Je_i)\\
&=e_i^*(\pi_LJ\widetilde{R}(X,Y)e_i)+(Je_i)^*(\pi_JJ\widetilde{R}(X,Y)Je_i)\\
&=\overline{g}(\pi_LJ\widetilde{R}(X,Y)e_i,e_i)+\overline{g}(\pi_JJ\widetilde{R}(X,Y)Je_i,Je_i),
\end{align*}
where the projections are included because $\{e_1,\ldots,e_n,Je_1,\ldots,Je_n\}$ is not an orthogonal basis.  
Hence, since $\tnabla J=0$ (as the connection is complex),
\begin{align*}
\widetilde{P}(X,Y)&=\overline{g}(J\pi_J\widetilde{R}(X,Y)e_i,e_i)-\overline{g}(J\pi_JJ^2\widetilde{R}(X,Y)e_i,e_i)\\
&=2\overline{g}(J\pi_J\widetilde{R}(X,Y)e_i,e_i)\\
&=2\tr_L(J\pi_J\widetilde{R}(X,Y)),
\end{align*}
so we have the final result.
\end{proof}

Notice that if $M$ is K\"ahler and $L$ is Lagrangian, $\xi_J$ coincides with the 1-form $\xi=\tr_L(J\pi_{\perp}\overline{\nabla}_X)$ defined in 
Section \ref{kahler_mean_curvature}, cf.~\eq{xi.trace.eq}. In Section \ref{ss:lagr_defs} we showed that if $L$ is Lagrangian then $\xi$ is 
directly related to the mean curvature vector field; for general totally real submanifolds the relationship is more complicated, cf.~(\ref{H.hook.omega.eq}). This is reflected by the fact that MCF has good properties for Lagrangians only in the K\"ahler setting, and it typically does not 
for totally reals. The next proposition will show that, given a general totally real submanifold, the proper quantity to consider is $\xi_J$, not $\xi$. Likewise, there is an appropriate replacement for $H$, which we now define.

Let us use the metric $\bar{g}$ to define the transposed operators
$$\pi_J^t:T_pM\rightarrow (T_pL)^\perp, \ \ \pi_L^t:T_pM\rightarrow (J(T_pL))^\perp.$$
Observe that $(J(T_pL))^\perp=J(T_pL)^{\perp}$ since $X\in (J(T_pL))^{\perp}$ if and only if for all $Y\in T_pL$,
$$\overline{g}(Y,JX)=-\overline{g}(JY,X)=0,$$
which means $JX\in (T_pL)^{\perp}$ and thus $X\in J(T_pL)^{\perp}$.  
Then, using the  tangential projection $\pi_T$ defined using $\overline{g}$, one may check that
$$\pi_T J\tnabla\pi_L^t:T_pL\times T_pL\rightarrow T_pL$$
is $C^{\infty}$-bilinear on its domain, so it is a tensor and its trace is a well-defined vector on $L$. We now set 
\begin{equation}\label{HJ.eq}
H_J:=-J(\mbox{tr}_L(\pi_T J\tnabla\pi_L^t)).
\end{equation}
This is a well-defined vector field on $L$. 

Let $\widetilde{T}$ denote the torsion of $\tnabla$:
$$\widetilde{T}(X,Y)=\tnabla_XY-\tnabla_YX-[X,Y].$$
We use it to define the vector field
\begin{equation}\label{TJ.eq}
T_J:=-\bar{g}(\pi_LJ\widetilde{T}(e_j,e_i),e_i)Je_j.
\end{equation}
Both $H_J$ and $T_J$ take values in the bundle $J(TL)$.

The following important result should be compared to Proposition \ref{H.hook.omega}.

\begin{prop}\label{HJ-prop} Let $\xi_J^\sharp$ denote the vector field on $L$ corresponding to the 1-form $\xi_J$ using the induced metric $g$. 
Then $\xi_J^\sharp=JH_J+JT_J$, so
$$\overline{\omega}(H_J+T_J,X)=\xi_J(X).$$
\end{prop}

\begin{proof}
This result follows from elementary computations.  
We calculate:
\begin{equation*}
\xi_J(X)=\overline{g}(J\pi_J\tnabla_Xe_i,e_i)=\overline{g}(J\pi_J\tnabla_{e_i}X+J\pi_J\widetilde{T}(X,e_i),e_i)
\end{equation*}
since $\pi_J[X,e_i]=0$.  Thus, as $J\pi_J=\pi_LJ$ by Lemma \ref{proj.lem} and $\tnabla$ is a complex connection we have that
\begin{align*}
\overline{g}(J\pi_J\tnabla_Xe_i,e_i)&=\overline{g}(\pi_LJ\tnabla_{e_i}X+\pi_LJ\widetilde{T}(X,e_i),e_i)\\
&=\overline{g}(\tnabla_{e_i}JX,\pi_L^{\rm t}e_i)+\overline{g}(J\widetilde{T}(X,e_i),\pi_L^{\rm t}e_i).
\end{align*}
Since $\overline{g}(JX,\pi_L^{\rm t}e_i)=\overline{g}(\pi_LJX,e_i)=0$ we see that 
\begin{align*}
\overline{g}(\tnabla_{e_i}JX,\pi_L^{\rm t}e_i)&=-\overline{g}(JX,\tnabla_{e_i}\pi_L^{\rm t}e_i)\\
&=\overline{g}(X,J\tnabla_{e_i}\pi_L^{\rm t}e_i)\\
&=\overline{g}(X,\pi_TJ\tnabla_{e_i}\pi_L^{\rm t}e_i)=g(X,JH_J).
\end{align*}
Replacing $X=\overline{g}(X,e_j)e_j$ since $e_1,\ldots,e_n$ is an orthonormal basis for $T_pL$, we deduce that
\begin{align*}
\overline{g}(J\widetilde{T}(X,e_i),\pi_L^{\rm t}e_i)&=\overline{g}(J\widetilde{T}(\overline{g}(X,e_j)e_j,e_i),\pi_L^{\rm t}e_i)\\
&=\overline{g}(X,e_j)\overline{g}(J\widetilde{T}(e_j,e_i),\pi_L^{\rm t}e_i)\\
&=\overline{g}(X,\overline{g}(J\widetilde{T}(e_j,e_i),\pi_L^{\rm t}e_i)e_j)\\
&=g(X,\overline{g}(\pi_LJ\widetilde{T}(e_j,e_i),e_i)e_j)\\
&=g(X,JT_J).
\end{align*}
We conclude that
$$\xi_J(X)=g(JH_J,X)+g(JT_J,X).$$
The result follows.
\end{proof}

\begin{remark}
The key point in the definitions of $\xi_J$ and $H_J$ is the identification of quantities containing $\widetilde{\nabla}$ which are adapted to totally real geometry and which exhibit tensorial behaviour analogous to the second fundamental form $\pi_\perp\overline{\nabla}$ used in standard Riemannian geometry.
\end{remark}

\subsection{Canonical choices in special cases}\label{ss:special}

The theory defined up to here depends on various choices: $J$, $\overline{g}$, $\tnabla$. We have tried to emphasize the role played by each of these structures on $M$. 

On a given manifold $M$ there is usually no canonical choice of such structures. Furthermore, even after making the choice of $J$ and $\overline{g}$, \textit{i.e.}~in an almost Hermitian manifold, there is no canonical choice of unitary connection $\tnabla$. Even Gauduchon's study of canonical connections \cite{Gau} leaves us with a 1-parameter family of connections to choose from.

However, there are two special cases where we can restrict the number of arbitrary choices made. We present these below, starting with the obvious one.

\paragraph{{\boldmath $M$} K\"ahler.} In this case $J$ and $\overline{g}$ are chosen so as to have very strong algebraic, geometric and analytic properties. There is then a canonical choice of unitary connection: we may use the Levi-Civita connection $\overline{\nabla}$. The advantage of this choice is that it is also torsion-free, leading to several simplifications in our formulae. Furthermore, (\ref{overline.rho.eq}) shows that 
$$\widetilde{P}(X,Y)=\overline{\omega}(\overline{R}(X,Y)\overline{e}_j,\overline{e}_j)=2\overline{\rho}(X,Y).$$

\paragraph{{\boldmath $M$} almost K\"ahler.} 
The Grassmannian of totally real $n$-planes in $M$ depends on the choice of $J$. Changing $J$ will produce a different set of totally real submanifolds. For example, a totally real submanifold may become partially complex under a change of $J$.

This may appear to be in contrast with Section \ref{ss:canonical_data}, where we characterized totally real planes $\pi$ in terms of the line bundle $K_M$. Indeed, it is known that complex line bundles are 
completely determined by their first Chern class $c_1$. Since this is an integral class, it is for example invariant under continuous deformations of $J$. 
However, notice that our characterization does not depend solely on the line bundle: it also depends on a particular pairing between the line bundle and $\pi$, 
\textit{i.e.}~on the identification of the line bundle with the space $K_M$ of $(n,0)$-forms, and this identification does depend on the choice of $J$.

Symplectic geometry provides a well-known framework within which $c_1$ is fixed: specifically, $c_1$ can be defined using any $J$ compatible with the given 
$\overline{\omega}$. From the almost Hermitian point of view, this is the realm of almost K\"ahler manifolds. Specifically, we start with a symplectic manifold 
$(M,\overline{\omega})$ and we add a choice of
a Riemannian metric $\overline{g}$ and an orthogonal almost complex structure $J$ 
such that $\overline{\omega}(X,Y)=\overline{g}(JX,Y)$
for all tangent vectors $X,Y$ on $M$.  

We let $\overline{\nabla}$ denote the Levi-Civita connection of $\overline{g}$.  It is well-known (see, for example, \cite{ApDr}) that
$$\overline{\nabla}_{JX}J=(\overline{\nabla}_XJ)J=-J(\overline{\nabla}_XJ).$$ 

In this setting all of Gauduchon's connections coincide, defining a canonical unitary connection known as the Chern connection.
Let $\widetilde{\nabla}$ denote the 
Chern connection and let $\widetilde{T}$ denote the torsion of $\tnabla$.  Specifically,
$$ \tnabla_XY= \overline{\nabla}_XY +\frac{1}{2}(\overline{\nabla}_XJ)(JY) $$
and
$$\widetilde{T}(X,Y)=\frac{1}{2}(\overline{\nabla}_XJ)(JY)-\frac{1}{2}(\overline{\nabla}_YJ)(JX).$$
Then
\begin{align*}
\widetilde{T}(JX,Y)&=\frac{1}{2}(\overline{\nabla}_{JX}J)(JY)-\frac{1}{2}(\overline{\nabla}_YJ)(J^2X)\\
&=-\frac{1}{2}J(\overline{\nabla}_XJ)Y+\frac{1}{2}J(\overline{\nabla}_YJ)(JX)\\
&=-J\widetilde{T}(X,Y),
\end{align*}
from which we deduce that $\widetilde{T}(JX,Y)=\widetilde{T}(X,JY)$, i.e.~the $(1,1)$ part of the torsion of $\tnabla$ vanishes (in fact, this characterizes the 
Chern connection amongst complex metric connections).

\section{The {\boldmath $J$}-volume functional}\label{s:Jvol}

Proposition \ref{HJ-prop} relates the Maslov 1-form to a vector field $H_J$. In the analogous Proposition \ref{H.hook.omega} $H_J$ coincided with the mean curvature vector field. We thus want to further investigate the geometric content of $H_J$. 
To this end we will use the variation formulae for the $J$-volume functional computed in \cite{LotayPacini2}. 

\subsection{The gradient of {\boldmath $\Vol_J$}}

Recall that the $J$-volume functional $\Vol_J$ is defined on the space $\cal{T}$ whose tangent space, at a given totally real submanifold $L\subseteq M$, is isomorphic to the space of vector fields of the form $JY$, where $Y$ is tangent along $L$. In computing the first variation of this functional it thus suffices to restrict to such vector fields. The following formula is proved in \cite{LotayPacini2}.

\begin{prop}\label{first.var.prop}
Let $\iota_t:L\rightarrow L_t\subseteq M$ be compact totally real submanifolds in an almost Hermitian manifold and let 
$\frac{\partial}{\partial t}\iota_t|_{t=0}=JY$ for $Y$ tangential.  Then 
$$\frac{\partial}{\partial t}\Vol_J(L_t)|_{t=0}=
-\int_L\overline{g}(JY,H_J+S_J)\vol_J$$
where $H_J$ is given by \eq{HJ.eq} and for $p\in L$ and an orthonormal basis $e_1,\ldots,e_n$ for $T_pL$ we have 
\begin{equation}\label{SJ.eq}
S_J=-\overline{g}(\pi_L\widetilde{T}(Je_j,e_i),e_i)Je_j.
\end{equation}

\end{prop}
It follows that, with respect to the Riemannian metric $G$ on $\TRL$ defined as $$G_L(JX,JY):=\int_L\overline{g}(X,Y)\vol_J$$ for $JX,JY\in T_L\TRL$, we see  $H_J+S_J$ is
the negative gradient of  $\Vol_J$.

\paragraph{Comparison with the Maslov form.}We now can compare the gradient of the $J$-volume with the vector field $-J\xi_J^\#=H_J+T_J$ defined by the Maslov form.  From \eq{HJ.eq}, \eq{TJ.eq} and \eq{SJ.eq} we see that
$$(H_J+S_J)-(H_J+T_J)=S_J-T_J=\overline{g}(\pi_L(J\widetilde{T}(e_j,e_i)-\widetilde{T}(Je_j,e_i)),e_i)Je_j,$$
so in a general almost Hermitian manifold the two objects are different. Futhermore, at this level the critical points of $\Vol_J$ do not appear to have particular geometric significance. 

When $M$ is an almost K\"ahler manifold endowed with the Chern connection then $\widetilde{T}(JX,Y)=-J\widetilde{T}(X,Y)$ so
$S_J=-T_J$. In the K\"ahler case, $\tnabla=\overline{\nabla}$ is the Levi-Civita connection and hence is torsion-free, \textit{i.e.}~$\widetilde{T}=0$, so $S_J=T_J=0$. We deduce the following important result.

\begin{thm}\label{thm:flows_coincide}
In K\"ahler manifolds, the negative gradient of the $J$-volume coincides with the vector field $-J\xi_J^\#$ defined by the Maslov form.
\end{thm}

In particular, in the K\"ahler setting this result allows us to transfer to either context any information already available for the other. For example, we can now characterize critical points of the $J$-volume functional as those for which $\Omega_J$ is parallel, and we discover that moving a submanifold in the direction $-J\xi_J^\#$ produces a monotone change in a certain quantity, namely the $J$-volume. 

\begin{remark} The condition that $\Omega_J$ is parallel does not imply $\rho_J$ is constant. Indeed, recall from \eqref{OmegaL.eq} that
\begin{equation*}
(\tnabla_X\Omega_J)(e_1,\ldots,e_n)=i\xi_J(X)\Omega_J(e_1,\ldots,e_n)=i\xi_J(X)\rho_J.
\end{equation*}
As the right hand side is imaginary, so the condition $\tilde{\nabla}\Omega_J=0$  affects only the imaginary part of the left hand side. We can rewrite the left hand side as
\begin{align*}
(\tnabla_X\Omega_J)(e_1,\ldots,e_n)&=\tnabla_X(\Omega_J(e_1,\ldots,e_n))-\sum_{i=1}^n\Omega_J(e_1,\ldots,\tnabla_Xe_i,\ldots,e_n)\\
&=X(\rho_J)-\sum_{i=1}^n\Omega_J(e_1,\ldots,\tnabla_Xe_i,\ldots,e_n).
\end{align*}
The variation $X(\rho_J)$ of $\rho_J$ is real, so it is not affected by $\Omega_J$ being parallel.
\end{remark}

Recall from Lemma \ref{Lag.Jvol.lem} that the $J$-volume coincides with the standard volume on Lagrangians and that, restricted to Lagrangians, the set of critical points of the $J$-volume and of minimal Lagrangians coincide.  In the appropriate context we can now improve this result by eliminating the Lagrangian hypothesis.

\begin{prop}\label{min.Lag.lim.prop}
Let $M$ be a K\"ahler--Einstein manifold with $\overline{\Ric}\neq 0$.  Then the set of critical points of the $J$-volume coincides with the set of minimal Lagrangian submanifolds.
\end{prop}

\begin{proof}
Let $\iota:L\rightarrow M$ be a critical point. According to Theorem \ref{thm:flows_coincide} it follows that $\xi_J=0$. Proposition \ref{xiJprop} shows that 
$\d\xi_J=\iota^*\overline{\rho}$. Since the Ricci form $\overline{\rho}=\lambda\overline{\omega}$ for some $\lambda\neq 0$, 
we see that $\iota^*\overline{\omega}=0$ and $\iota$ is Lagrangian. Lemma \ref{Lag.Jvol.lem} now shows that 
$H=0$. Conversely, if $\iota$ is minimal Lagrangian then it is a critical point by Lemma \ref{Lag.Jvol.lem}.
\end{proof}

Proposition \ref{min.Lag.lim.prop} shows that replacing the standard volume with the $J$-volume serves to filter out all other critical points, leaving only the minimal Lagrangians. We will see in Section \ref{s:calibration} that this result is in marked contrast to the Ricci--flat case, where we can have non-Lagrangian critical points for $\Vol_J$ (called STR submanifolds).

\begin{remark}
The stability of critical points of the $J$-volume turns out to be an interesting issue, studied in \cite{LotayPacini2}. The second variation formula shows that, when $M$ is K\"ahler with negative Ricci curvature, not only is it true that critical points are automatically stable, but also that the $J$-volume is strictly convex with respect to a certain notion of geodesics on the infinite-dimensional space of totally real submanifolds. An application of these results appears in \cite{LotayPacini3}.
 \end{remark}

\subsection{Comments on the negative gradient flow}\label{ss:comments_Jflow}
Given the above, it seems reasonable to study the \textit{$J$-mean curvature flow} (J-MCF) of totally real submanifolds, defined as the negative gradient flow of $\Vol_J$:
\begin{equation*}
 \frac{\partial\iota_t}{\partial t}=H_J[\iota_t]+S_J[\iota_t].
\end{equation*}
In certain situations one might expect this to coincide with standard MCF:  
for example, if both flows preserved the Lagrangian condition, Lemma \ref{Lag.Jvol.lem} would imply that they coincide on Lagrangians. However, in generic (\textit{i.e.}~non KE) almost K\"ahler manifolds there is no reason why the Lagrangian condition should be preserved by these flows. 

More importantly, in Section \ref{s:shorttime}, we will see that the operator $H_J$ is highly degenerate, making the existence theory for J-MCF rather challenging. It is an appealing idea to try to rely on the existence theory of the (less degenerate) MCF to obtain results for J-MCF: we give examples of this line of thought in $\S$\ref{s:shorttime}. Overall, however, it seems difficult to prove strong results for this flow.

In the next section we will define and study an alternative flow (the Maslov flow) in terms of the Maslov $1$-form which turns out to have better analytic and geometric properties. It follows from Theorem \ref{thm:flows_coincide} that in the K\"ahler setting  Maslov flow and $J$-MCF coincide, so any results concerning the Maslov flow will hold also for  $J$-MCF.  

In any case it is interesting to speculate on the properties of $J$-MCF. Although this flow seeks to find totally real submanifolds which minimize the $J$-volume, it may happen that the flow converges to a submanifold for which $\Vol_J$ is zero, \textit{i.e.}~to a partially complex submanifold. For example, if we start with a partially complex submanifold $L'$  and perturb it slightly to become a totally real $L$, then $\Vol_J(L)$ should be very small and hence it 
should flow back to $L'$ under $J$-MCF. In other words, one should expect that  $J$-MCF can leave the totally real submanifolds $\mathcal{T}$ and reach the ``boundary'' of $\mathcal{T}$. Given that this behaviour would interrupt the flow, we should incorporate it into any notion of ``singularity formation'' for $J$-MCF.

\section{The Maslov flow}\label{s:MF}

Assume we are given a unitary connection $\tnabla$ on an almost Hermitian manifold $(M,J,\bar{g})$. Consider the induced connection $\tnabla$ on $K_M$: this allows 
us to define the Maslov form $\xi_J[\iota]$ of any totally real immersion $\iota$. We can use the induced metric $g$ on $L$ to view $\xi_J[\iota]$ as a tangent 
vector field so that $J\iota_*\xi_J^{\sharp}[\iota]$ is a section of $J(TL)$. 
The \emph{Maslov flow} for a family of immersions $\iota_t:L\rightarrow M$ such that $\iota_0=\iota$ is given by:
\begin{equation}\label{MF.eq}
\frac{\partial}{\partial t}\iota_t=-J\iota_{t*}\xi_J^{\sharp}[\iota_t]=H_J[\iota_t]+T_J[\iota_t].
\end{equation}
Stationary points for the Maslov flow are immersions for which $\xi_J=0$ or equivalently $\tnabla\Omega_J=0$, \textit{i.e.}~$\Omega_J$ is a parallel section. In 
particular the induced connection is flat as $\d\xi_J=0$.  Observe that if $\iota_t$ satisfies Maslov flow then so does $\iota_t\circ\varphi$ for 
any $\varphi\in\Diff(L)$, so Maslov flow defines a flow on submanifolds in $\TRL$. 

The short-time existence of Maslov flow is highly non-trivial and will be discussed in Section \ref{s:shorttime}. In the course of this section we will thus simply assume that solutions exist and concentrate instead on their geometric properties. We will show that this flow is particularly interesting in the context of almost K\"ahler manifolds where its properties are strongly analogous to those valid for standard MCF, when we restrict the latter to Lagrangians in K\"ahler ambient spaces.

Let us start by reviewing an interesting ambient flow for almost K\"ahler manifolds, introduced by Streets and Tian.

\subsection{Introduction to symplectic curvature flow}\label{ss:SCF}

Streets--Tian \cite{StrTian} consider the following flow of the almost K\"ahler structure on $M$ (up to a factor of $\frac{1}{2}$):
\begin{equation}\label{SCF.eq}
\frac{\partial}{\partial t}\overline{\omega}=-\frac{1}{2}\widetilde{P},\quad\frac{\partial}{\partial t}J=-\frac{1}{2}\overline{\nabla}^*\overline{\nabla}J+\frac{1}{2}\overline{\mathcal{N}}+\frac{1}{2}\overline{\mathcal{R}}
\end{equation}
where, if $\{\bar{e}_1,\ldots,\bar{e}_{2n}\}$ is a local orthonormal frame on $M$,
$$\overline{g}(\overline{\mathcal{N}}(X),Y)=\overline{g}\big((\overline{\nabla}_{\overline{e}_k}J)JX,(\overline{\nabla}_{\overline{e}_k}J)Y\big);$$
$$\overline{g}(\overline{\mathcal{R}}(X),Y)=\overline{\Ric}(JX,Y)+\overline{\Ric}(X,JY).$$
Since $\widetilde{P}$ is closed and (up to a constant multiple) represents the first Chern class of $M$, this symplectic curvature flow (SCF) preserves the closedness of the 2-form $\overline{\omega}$
and is a natural extension of K\"ahler--Ricci flow: in the K\"ahler case we see immediately that the flow reduces to K\"ahler--Ricci flow. The flow of $J$ ensures that the compatibility condition between $\overline{g}$, $J$ and $\overline{\omega}$ is 
preserved.  In particular, we have that $\overline{g}(\overline{\nabla}^*\overline{\nabla}JX-\overline{\mathcal{N}}(X),Y)$ is, up to a constant factor, equal to the $(2,0)+(0,2)$ part of
$\widetilde{P}(X,Y)$.  Moreover, the induced flow on the metric $\overline{g}$ is Ricci flow plus some lower order terms which vanish in the K\"ahler setting.

The stationary solutions (and the expanding and shrinking solitons) of SCF are not fully understood: namely 
solutions to 
\begin{equation*}\label{SCF.soliton.eq}
\widetilde{P}=2\lambda\overline{\omega}\quad\text{and}\quad \overline{\nabla}^*\overline{\nabla}J-\overline{\mathcal{N}}-\overline{\mathcal{R}}=0
\end{equation*}
for some constant $\lambda$.  These can all be viewed as stationary points of what one might call normalized symplectic curvature flow, namely:
\begin{equation*}\label{NSCF.eq}
\frac{\partial}{\partial t}\overline{\omega}=-\frac{1}{2}\widetilde{P}+\lambda\overline{\omega},\quad\frac{\partial}{\partial t}J=-\frac{1}{2}\overline{\nabla}^*\overline{\nabla}J+\frac{1}{2}\overline{\mathcal{N}}+\frac{1}{2}\overline{\mathcal{R}}.
\end{equation*}
  K\"ahler--Einstein metrics clearly solve these equations, but non-trivial solutions are also possible: for example there exist compact non-K\"ahler solitons for SCF with constant $J$, cf.~\cite{Pook}. Analogous examples do not exist in dimension 4: any such ``static'' solution is necessarily K\"ahler--Einstein \cite[Corollary 9.5]{StrTian}. The metrics obtained in \cite{Pook} are not Einstein;  actually, there is a conjecture due to Goldberg stating that any compact almost K\"ahler manifold whose metric is Einstein is necessarily K\"ahler--Einstein.

\subsection{Maslov flow and symplectic curvature flow}

As we have seen, in the almost K\"ahler setting we have a canonical choice
 of complex 
Hermitian connection (the Chern connection), so the Maslov flow provides a canonical way to deform totally real submanifolds.

We now show that the Maslov flow is ``compatible'', in a precise way, with
 symplectic curvature flow. Notice that the definitions of the two flows are
  completely independent of each other, and each is based on its own specific 
  set of geometric considerations. This means  the compatibility was not ``built 
  into'' the definitions: it reveals something interesting about both flows.

\begin{thm}\label{MSCF.thm}
Let $\iota:L\rightarrow M$ be a totally real submanifold of an almost 
K\"ahler manifold $(M,J,\overline{\omega})$.  Suppose that $(J_t,\overline{\omega}_t)$ satisfies symplectic curvature flow as in \eq{SCF.eq}
with $\overline{\omega}_0=\overline{\omega}$ and 
$\iota_t:L\rightarrow (M,J_t,\overline{\omega}_t)$ satisfies Maslov flow as in \eq{MF.eq} with $\iota_0=\iota$.
Then
$$\frac{\partial}{\partial t}\iota_t^*\overline{\omega}_t=0$$
for all $t>0$ for which the flows exist;
i.e.~the 2-form $\omega_t=\iota_t^*\overline{\omega}_t$ is preserved along the coupled flow.  
\end{thm}

\begin{proof}  
Recall the equation for the Maslov flow
$$\frac{\partial}{\partial t}\iota_t=H_{J}+T_J,$$
where $H_{J}$ and $T_J$ are computed using the Chern connection on $(M,J_t,\overline{\omega}_t)$ via \eq{HJ.eq} and \eq{TJ.eq}. Since $\overline{\omega}_t$ remains closed along SCF, we calculate 
\begin{align*}
\frac{\partial}{\partial t}\iota_t^*\overline{\omega}_t&=\iota_t^*\mathcal{L}_{(H_J+T_J)}\overline{\omega}_t+\iota_t^*\frac{\partial}{\partial t}\overline{\omega}_t\\
&=\iota_t^*\d((H_J+T_J)\lrcorner\overline{\omega}_t)-\frac{1}{2}\iota_t^*\widetilde{P}_t=0
\end{align*}
by Propositions \ref{xiJprop} and \ref{HJ-prop}.  
\end{proof}

\begin{remark} Our formulae for totally real submanifolds do not require the almost K\"ahler condition: in particular, $\d\xi_J$ is 
always $\frac{1}{2}\tilde{P}$.  The only place in the proof where we make specific use of 
the fact that $\overline{\omega}_t$ is closed is when we simplify $\mathcal{L}_{(H_J+T_J)}\overline{\omega}_t$ using Cartan's formula.  This step makes it 
non-obvious to see how similar results could hold in the more general setting of almost Hermitian manifolds,  
coupling Maslov flow with ambient flows available in the literature which include additional terms in the evolution of the almost Hermitian structure.
 \end{remark}

 In the special case of Lagrangian submanifolds, this allows us to generalize Theorem \ref{main.LMCF.thm} to almost K\"ahler manifolds, as follows.
 
\begin{cor}\label{LSCF.cor}
Let $\iota:L\rightarrow M$ be a Lagrangian submanifold of an almost K\"ahler manifold.  
If $M$ evolves by 
symplectic curvature flow and $L$ evolves 
by Maslov flow then $\iota_t:L\rightarrow M$ is Lagrangian with respect to $\overline{\omega}_t$ for all $t$.
\end{cor}

\paragraph{Concluding remarks.}Let us pause to compare Theorem \ref{MSCF.thm} with Theorem \ref{main.LMCF.thm}. 
To simplify the comparison, let us start by assuming that $M$ is K\"ahler. In this
 case SCF coincides with KRF and the Maslov flow coincides with 
 $J$-MCF 
(which is MCF on Lagrangians), so the two theorems are formally analogous. Even here, Theorem \ref{MSCF.thm} is significantly stronger than the other, 
in two respects:
\begin{itemize}
 \item it does not assume that $L$ is compact;
 \item it proves that the 2-form $\omega_t$ is preserved \emph{pointwise}, regardless of its initial value; in other words, it applies to all totally real submanifolds, rather than only to Lagrangians.
\end{itemize}
Why is this true? Recall our concluding remarks at the end of Section \ref{ss:existence_lagr_mcf}. On Lagrangians the $J$-volume and the standard volume coincide; our new strategy is to choose the $J$-volume extension of this functional to $\mathcal{T}$, rather than the standard volume functional. This leads to the uniformly simple formulae 
\begin{equation}\label{simple.eq}
\xi_J=\iota^*\overline{\omega}(H_J,\cdot), \ \ \d\xi_J=\frac{1}{2}\iota^*\tilde{P},
\end{equation}
valid for both Lagrangians and totally real submanifolds. In turn these lead to an ODE on $\omega_t$ rather than a parabolic inequality such as (\ref{eq:parabolic_inequality}), so we can eliminate the maximum principle argument and the related compactness assumption.

Equations \eqref{simple.eq} rely also on the K\"ahler assumption, which cancels the quantity $T_J$ which should have appeared. 
The generalization to almost K\"ahler manifolds brings $T_J$ back, and thus requires replacing the gradient flow with the Maslov flow. This explains why the result, in its most general form, concerns the coupling of SCF with Maslov flow, rather than with  $J$-MCF.

On the other hand, even in the K\"ahler case, replacing $\Vol_g$ with $\Vol_J$ leads to a more degenerate operator governing the  flow. This makes the existence theory much more complicated: we will discuss this  at length in Section \ref{s:shorttime}.

To close, it may be useful to emphasize a basic difference between $J$-MCF (or MCF) and the Maslov flow. The former is generated by a functional which is invariant under reparametrization. This implies that the corresponding flow is invariant: reparametrization adds tangential motions to the flow, affecting its direction in $TM$ but not the image submanifolds. In other words, from the point of view of the functional, our definition of the $J$-mean curvature vector field is not particularly canonical (and indeed, it depends on the choice of an $L^2$-type metric).
Reparametrization, however, will usually change $\omega$ (except in the case of Lagrangians). Preserving $\omega$ can be thought of as analogous to gauge-fixing: it is a strong condition on the immersions, rather than on the image submanifolds. In this sense we can think of the Maslov flow as a gauge-fixed version of the gradient flow. To define this flow, the idea of replacing the more classical transverse space $TL^\perp$ with $J(TL)$ is vital. Given that the two flows differ at most by only lower-order terms (defined via the torsion), we start to notice the geometric importance of a careful choice of such terms.

\subsection{Flows in K\"ahler--Einstein manifolds}\label{ss:KE_flows}

Recall the difference between Theorem \ref{LMCF.KE.thm} and Theorem \ref{LMCF.KRF.thm}: the former manages to avoid the need for an ambient flow by assuming the KE condition. There is however an alternative way of thinking about this situation, as in Corollary \ref{cor:normalized}: if $M$ is KE and we couple MCF with the normalized KRF, the ambient remains static so we exactly recover Theorem \ref{LMCF.KE.thm}.

Let us apply this same game to Theorem \ref{MSCF.thm}, using the normalized version of SCF introduced in Section \ref{ss:SCF}. The same proof then leads to the following.

\begin{thm}\label{thm:MSCF_normalized}
 Let $\iota:L\rightarrow M$ be a totally real submanifold of an almost 
K\"ahler manifold $(M,J,\overline{\omega})$.  Suppose that $(J_t,\overline{\omega}_t)$ satisfies the normalized symplectic curvature flow with respect to the constant $\lambda$ and that 
$\iota_t:L\rightarrow (M,J_t,\overline{\omega}_t)$ satisfies Maslov flow.

Then, for all $t>0$ for which the flow exists,
$$\frac{\partial}{\partial t}\iota_t^*\overline{\omega}_t=\lambda\iota_t^*\overline{\omega}_t;$$
i.e.~the 2-form $\omega_t=\iota_t^*\overline{\omega}_t$ changes exponentially.
\end{thm}
 
Let us apply Theorem \ref{thm:MSCF_normalized} to the case in which $M$ is KE with Ricci form $\overline{\rho}=\lambda\overline{\omega}$. Here the KE condition exactly counteracts the ambient flow, leaving the structure on $M$ static. Thus, we can judge the effect of the Maslov flow (equivalently, of the $J$-mean curvature flow) on $\omega_t$, taken as a single equation rather than as part of a coupled system: we see that $\omega_t$ changes exponentially.

\begin{prop}\label{limit.prop}
Let $\iota:L\rightarrow M$ be a totally real submanifold in a K\"ahler--Einstein manifold $M$ and let $\iota_t$ denote the solution to the Maslov flow with $\iota_0=\iota$, for some maximal time interval $[0,T)$.
\begin{itemize}
 \item Assume  $\overline{\Ric}\geq 0$.  If $\iota:L\rightarrow M$ is 
not Lagrangian then it cannot become Lagrangian under Maslov flow. 
\item Assume $\overline{\Ric}<0$ and that $\iota_t(L)=L_t$ converges to a smooth submanifold $L_T$ of the same dimension as $t\rightarrow T$. Then $L_T\in \mathcal{T}$ if and only if $T=\infty$; 
in this case $L_T=L_{\infty}$ is a minimal Lagrangian submanifold. Conversely, $L_T$ is partially complex only if $T<\infty$.  
\end{itemize}

\end{prop}

\begin{proof}
Theorem \ref{thm:MSCF_normalized} shows that $\omega_t=e^{\lambda t}\omega_0$.
Assume $\lambda\geq 0$. If $\iota$ is not Lagrangian, there exist tangent vectors $X,Y$ on $L$ such that $\omega_0(X,Y)\neq 0$. Then 
$$|\omega_t(X,Y)|=e^{\lambda t}|\omega_0(X,Y)|\geq |\omega_0(X,Y)|$$ for all 
$t\geq 0$. This quantity is monotone increasing or constant, so in the limit it cannot vanish. 

Now assume $\lambda<0$. If $\iota_t$ converges within $\mathcal{T}$ then $T=\infty$ otherwise the flow could be continued as $L_T$ is smooth so the time interval would not be maximal. Now assume $T=\infty$ and $L_t$ converges to some $L_\infty$. Since $L_\infty$ is smooth of the same dimension, the induced metric $g_\infty$ is smooth. Then $|\omega_t|_{g_t}=e^{\lambda t}|\omega_0|_{g_t}$ and 
$$\lim_{t\rightarrow\infty}|\omega_0|_{g_t}=|\omega_0|_{g_{\infty}}$$ 
is finite.  Hence, as $\lambda<0$ we 
see that $|\omega_{\infty}|_{g_{\infty}}=0$ and thus $L_\infty$ is Lagrangian, so in particular $L_\infty\in\mathcal{T}$. It is clear that $L_\infty$ must also be stationary for the Maslov flow, thus minimal. Notice that this agrees with  Proposition \ref{min.Lag.lim.prop}.
\end{proof}

Loosely speaking, Proposition \ref{limit.prop} says that the set $\mathcal{L}$ of Lagrangian submanifolds is an unstable subset of $\mathcal{T}$ for the Maslov flow when $\overline{\Ric}\geq 0$; it is an attractor for the 
flow when $\overline{\Ric}<0$.    The latter statement makes clearer the strict stability of the critical points in negative KE manifolds.
  
Proposition \ref{limit.prop} also shows that if we are in a Ricci-positive KE manifold $M$, where the only critical points are minimal Lagrangian by 
Proposition \ref{min.Lag.lim.prop}, and we start with a non-Lagrangian totally real submanifold $L$ then the Maslov flow cannot converge.  
One possibility is that $L$ becomes partially complex so that the Maslov flow becomes undefined; i.e.~the flow reaches the ``boundary'' of the space $\mathcal{T}$ of
totally real submanifolds in $M$.
   
\subsection{Relation to other Lagrangian flows}\label{ss:otherflows}

To conclude this section we make some observations relating the Maslov flow to other generalisations of Lagrangian MCF which have appeared in the recent literature. We emphasize however that the papers in question consider only Lagrangians, while the Maslov flow applies to any totally real submanifold.

Let $L$ be  a Lagrangian in an almost K\"ahler manifold $M$. Here
we showed that $S_J=-T_J$ for any totally real, and hence by Lemma \ref{Lag.Jvol.lem}  on $L$ we have
$H_J+S_J=H_J-T_J=H$.   
Therefore, on $L$,
$$H_J+T_J=H+2T_J=H+2\overline{g}(J\widetilde{T}(e_j,e_i),e_i)Je_j=\hat{H},$$
where $\hat{H}$ is the ``generalized mean curvature vector'' 
on Lagrangians defined in \cite{SmoWang} (applied to the Chern connection).  We thus see that Maslov flow is a natural 
extension of the generalized Lagrangian mean curvature flow introduced in
\cite{SmoWang}.  
However, in \cite{SmoWang}, they also extend $\hat{H}$ to totally reals as a \emph{normal} vector field, so the Maslov flow is not the $\hat{H}$-flow except on Lagrangians.

Now let  $\widetilde{H}=\pi_{\perp}\tnabla_{e_i}e_i$. Since $H_{J}+S_J=H$ and $H_J=\widetilde{H}$  on $L$, we see that 
$$H=H_J+S_J=\widetilde{H}+\overline{g}(J\widetilde{T}(e_j,e_i),e_i)Je_j,$$
so that the torsion terms 
$$\overline{g}(J\widetilde{T}(e_j,e_i),e_i)Je_j=H-\widetilde{H}.$$
We can therefore deduce that $T_J=\widetilde{H}-H$ so that, on $L$, 
$H_{J}+T_J=2\widetilde{H}-H$.  
Since $M$ is almost K\"ahler, the Chern connection is related to the Levi-Civita connection by 
$$\tnabla_XY=\overline{\nabla}_XY +\frac{1}{2}(\overline{\nabla}_XJ)(JY)=\frac{1}{2}\overline{\nabla}_XY-\frac{1}{2}J\overline{\nabla}_X(JY),$$
and so we see that, on $L$, 
$$H_{J}+T_J=2\widetilde{H}-H=-\pi_{\perp}J\overline{\nabla}_{e_i}(Je_i).$$
When $M$ also has $c_1=0$, this is referred to as the ``complex mean curvature vector'' in \cite{Bar}.  So we see that
 the Maslov flow also generalizes this work.

 \section{Short-time existence}\label{s:shorttime}

The existence of short-time solutions to parabolic equations is a standard fact. Geometrically defined flows, however, are usually not parabolic: this is related to the fact that the underlying operators are invariant with respect to some large group of transformations, generally known as the ``gauge group''.

Consider the case of MCF for immersions $\iota:L\rightarrow (M,\overline{g})$. The volume functional is invariant under reparametrizations, \textit{i.e.}~under the action of the group $\mbox{Diff}(L)$. This is reflected in the fact that, given $\phi\in\mbox{Diff}(L)$, the mean curvature vector has the property $H[\iota\circ\phi]=H[\iota]\circ\phi$. The non-ellipticity of $H[\iota]$, viewed as a second order operator on $\iota$, is another manifestation of this fact. Its symbol with respect to the generic non-zero 1-form $\zeta\in T_x^*L$,
\begin{equation*}
\sigma(H[\iota])_x(\zeta):T_{\iota(x)}M\rightarrow T_{\iota(x)}M,
\end{equation*}
is only a semi-positive, rather than positive, endomorphism: its kernel is $n$-dimensional, given by the subspace $\iota_*(T_xL)$. It is thus not immediately clear that MCF admits short-time solutions. This issue is generally resolved in two ways: either restricting to normal variations, \textit{i.e.}~working transversely to the gauge group, or via a standard argument known as ``DeTurck's trick'', following \cite{DeTurck} (though it was known before in other contexts).

In general, DeTurck's trick consists of (i) modifying the operator to make it elliptic, (ii) solving the corresponding parabolic equation using standard theory, and (iii) showing that this solution can be modified to obtain a solution to the original equation. Clearly, this final modification must be built \textit{ad hoc} for the specific flow, in a manner determined by the gauge group.

Uniqueness of the solution of such ``weakly parabolic'' equations is also an issue; an appropriate argument must be found for each case.

$J$-MCF has the same invariance property as MCF so one should expect it to have a degenerate symbol, as above. The Maslov flow differs from $J$-MCF only by torsion terms, which are first order. This implies they have the same symbol, determined by the operator $H_J$. They thus share the same degeneracies.

In computing this symbol, however, we will see that the kernel is much larger than expected purely from invariance under reparametrization. In particular, restricting to transverse variations will not suffice to obtain a parabolic equation. 

In special settings these equations do, nonetheless, admit solutions. For example, assume the ambient manifold is KE. According to Theorem \ref{LMCF.KE.thm} MCF, applied to Lagrangian initial data, produces a curve of Lagrangian submanifolds. Lemma \ref{Lag.Jvol.lem} shows this family automatically solves J-MCF, which thus admits a solution even though it is much more degenerate than MCF. We now apply Theorem \ref{thm:flows_coincide} to conclude that the Maslov flow is also solvable. Analogously, we could use Theorem \ref{LMCF.KRF.thm} to prove that the coupled systems J-MCF+KRF and MF+KRF admit solutions for Lagrangian initial data in any K\"ahler manifold. Notice however that this method does not prove that the Maslov flow or J-MCF admit solutions when considered as self-standing equations.  

In some sense we can view this method as an analogue of DeTurck's trick: we perturb the degenerate operator $H_J$ to the less degenerate $H$, obtain existence, then argue that the solutions coincide in certain situations: cf.~the second Remark following Corollary \ref{Lag.MF.exist.cor} for a similar (though slightly more involved) argument. However, this method would certainly not extend to general totally real submanifolds, nor to more general ambient spaces in which MCF does not preserve Lagrangians. It is thus necessary to find alternative ways to deal with the extra degeneracies of $H_J$.
 
For the above reasons we now present another method for proving short-time existence, based on the Nash--Moser implicit function theorem as formalized by Hamilton \cite{Hamilton}. This method completely bypasses the properties of MCF, allowing us to extend the above existence results to a wider category of ambient spaces: SCF solitons. 
The key ingredient will be the following, which we may immediately deduce from Propositions \ref{xiJprop}-\ref{HJ-prop}:
\begin{equation}\label{eq:integrability}
 \d(\iota^*\overline{\omega}(H_J[\iota]+T_J[\iota],\cdot))=\frac{1}{2}\iota^*\widetilde{P}.
\end{equation}
Notice that the third order operator in $\iota$ on the left-hand side equals a first order one on the right, so \eqref{eq:integrability} is clearly a strong condition.
    We already used \eqref{eq:integrability} in Section \ref{s:MF} for geometric purposes, to show that any initial tensor $\omega:=\iota^*\bar{\omega}$ is preserved during the Maslov flow coupled with symplectic curvature flow. Here instead we will show that it has analytic consequences: specifically, equation \eqref{eq:integrability} enables us to deal with the extra degeneracies in $H_J+T_J$ by providing the essential ``integrability condition'' required to implement the results in \cite{Hamilton} to prove short-time existence. 

 This approach has several consequences.
\begin{itemize}
\item Our method strongly relies on the properties of the full operator, not only on its highest order terms. It thus applies only to the Maslov flow, confirming that the existence of solutions to the $J$-mean curvature flow seems more difficult, cf.~Section \ref{ss:comments_Jflow}. 
\item It indicates that the coupling with SCF might be an important ingredient in the general existence theory for the Maslov flow. Equivalently, it shows the relevance of the preserved quantity $\omega$ for the existence theory.
\item It indicates a new role for (\ref{eq:integrability}), thus also for its classical counterpart (\ref{eq:d.H.hook.omega}).
\end{itemize}
Let us start by computing the symbol of the operator $\iota\mapsto H_J[\iota]$, viewed as a second-order operator on the space of totally real immersions.

\begin{prop}\label{p:symbol}
Given $x\in L$, $\zeta\in T^*_xL$ and a totally real immersion $\iota$, the symbol of the operator $H_J$ is
\begin{equation*}
\sigma(H_J[\iota])_x(\zeta):T_{\iota(x)}M\rightarrow T_{\iota(x)}M, \ \ Z\mapsto\overline{g}(J\iota_*\zeta^\#,\pi_JZ) J\iota_*\zeta^\#.
\end{equation*}
The kernel of this map is $\iota_*(T_xL)\oplus \left(\langle J\iota_*\zeta^\#\rangle^\perp\cap J\iota_*(T_xL)\right)$, and thus has dimension $2n-1$. Further, $Z:=J\iota_*\zeta^\#$ is an eigenvector, with eigenvalue $|\zeta|^2$.
\end{prop}
\begin{proof}
 Choosing local coordinates on $L$, let us identify $\partial_i$ with $\iota_*(\partial_i)$ and write 
 \begin{equation*}
  JH_J[\iota]=\mbox{tr}_g(\pi_T J\tnabla\pi^t_L)=g^{ij}\pi_TJ\tnabla_{\partial_i}\pi^t_L(\partial_j)=a^k \partial_k.
 \end{equation*}
We can compute the coefficents $a_k$ explicitly by noticing, in general, that if $v=a^k\partial_k$ then $g(v,\partial_l)=a^k g_{kl}$ thus $a^k=g^{kl}g(v,\partial_l)$. It follows that
\begin{align*}
 H_J[\iota]&=-g^{kl}g^{ij}g(\pi_T J\tnabla_{\partial_i}\pi^t_L (\partial_j), \partial_l) J\partial_k\\
 &=-g^{kl}g^{ij}\overline{g}(J\tnabla_{\partial_i}\pi^t_L (\partial_j), \partial_l) J\partial_k\\
 &=-g^{kl}g^{ij}\overline{g}(\partial_j, \pi_L\tnabla_{\partial_i}J\partial_l)J\partial_k\\
 &=g^{kl}g^{ij}\overline{g}(J\partial_j, \pi_J\tnabla_{\partial_i}\partial_l)J\partial_k.
\end{align*}
If we choose $\partial_i$ to be orthonormal in $x$, this expression simplifies to
\begin{equation*}
 H_J[\iota]=\overline{g}(J\partial_j,\pi_J\tnabla_{\partial_j}\partial_k)J\partial_k.
\end{equation*}
We now need to linearize this operator. We can identify any variation of $\iota$ as a vector field $Z$ and, in terms of our local coordinates, we can write $\tnabla$ as the standard differential plus lower order terms. Notice that $\pi_J$ is also a first order operator on $\iota$, so $H_J$ is quasi-linear; in particular, $\pi_J$ does not contribute to the second order terms of the linearization. Thus, up to lower order terms, the linearized operator is
\begin{equation*}
Z\mapsto\overline{g}(J\partial_j,\pi_J\frac{\partial^2 Z}{\partial_j\partial_k})J\partial_k.
\end{equation*}
This shows that the symbol with respect to $\zeta=\zeta_idx^i$ is
\begin{equation*}
 \sigma(H_J[\iota])_x(\zeta):Z\mapsto\overline{g}(J\partial_j,\pi_J Z)\zeta_j\zeta_k J\partial_k,
\end{equation*}
proving the claim.
\end{proof}

Proposition \ref{p:symbol} shows that $H_J$ is a particularly degenerate operator. However, Propositions \ref{xiJprop}--\ref{HJ-prop} show that $H_J+T_J$ satisfies the differential identity \eqref{eq:integrability}.
To understand this identity better, let us define the first order operator in two variables 
\begin{equation}\label{eq:Liota}
(\iota,Z)\mapsto L_\iota(Z):=\d(\iota^*\overline{\omega}(Z,\cdot)),
\end{equation}
which is linear in $Z$. For fixed $\iota$, the operator $Z\mapsto L_\iota(Z)$ has symbol
\begin{equation*}
 \sigma(Z\mapsto L_\iota(Z))_{|x}(\zeta):T_{\iota(x)}M\rightarrow \Lambda^2(T_x^*L),\ \ W\mapsto \zeta\wedge\iota^*\overline{\omega}(W,\cdot),
\end{equation*}
so its kernel is the ($n+1$)-dimensional space $J(\iota_*(T_xL))^\perp\oplus\langle J\iota_*\zeta^\#\rangle$.

Rewriting (\ref{eq:integrability}) as $2L_\iota(H_J[\iota]+T_J[\iota])=\iota^*\widetilde{P}$, we see that the composition of symbols on the left-hand side must vanish, and thus $$\sigma(Z\mapsto L_\iota(Z))\circ \sigma(\iota\mapsto H_J[\iota])=0.$$ We can interpret this as a constraint on the dimension of the image of the symbol of $H_J$. 
The key observation, however, is that $\langle J\iota_*\zeta^\#\rangle$ is precisely the positive eigenspace of the symbol of $H_J$.

We thus find ourselves in a situation very closely related to one introduced in \cite[$\S$5]{Hamilton}: equation (\ref{eq:integrability}) corresponds to Hamilton's ``integrability condition''. In order to apply that theory, however, it is necessary to ``linearize'' the setting of the problem, rephrasing the space of immersions into $M$ as a space of sections of a fixed vector bundle. We can then view the non-linear operator $H_J+T_J$ as a non-linear operator acting on these sections. We now review a standard way to achieve this. The bottom line will be that the two formulations are equivalent; indeed, we will ultimately work in terms of the original formulation so as to avoid the proliferation of pull-back operations.

\paragraph{The problem, reformulated.} Let $\overline{\exp}$ be the exponential map on $(M,\bar{g})$. Given a totally real immersion $\iota_0$, consider the diffeomorphism
\begin{equation*}
F:\mathcal{U}\subseteq TL\rightarrow \mathcal{V}\subseteq M,\ \ F(x,X):=\overline{\exp}_{|\iota_0(x)}(J\iota_{0*}(X))                                                                    
\end{equation*}
where $\mathcal{U}$ is an open neighbourhood of the zero section and $\mathcal{V}$ is an open neighbourhood of $\iota_0(L)$. Let $G:=F^{-1}$ be the inverse diffeomorphism. The properties of $\overline{\exp}$ imply that 
\begin{equation*}
 \frac{\partial F}{\partial X}_{|(x,0)}(\cdot)=J\iota_{0*}(\cdot):T_xL\rightarrow J\iota_{0*}(T_xL).
\end{equation*}
This implies that
\begin{equation*}
 G_{*|\iota_0}(H_J[\iota_0])=-\iota_{0*}^{-1}(JH_J[\iota_0]).
\end{equation*}
Set $X:=G(\iota)$, so that $\iota=F(X)$. Then
\begin{align*}
 \frac{\partial X}{\partial t}&=G_{*|\iota}\left(\frac{\partial\iota}{\partial t}\right)=G_{*|\iota}(H_J[\iota])=\widetilde{H}_J[X]\\
 X_{|t=0}&=G(\iota_0)=0,
\end{align*}
where we set $\widetilde{H}_J[X]:=G_{*|F(X)}(H_J[F(X)])$.

We now want to study the symbol of the operator $X\mapsto \widetilde{H}_J[X]$. Notice that $F$, respectively $G$, is defined pointwise, so it is of order zero in $X$, respectively $\iota$. This implies that 
\begin{equation*}
 \sigma(G_{*|\iota}(H_J[\iota]))=G_{*|\iota}\sigma(H_J[\iota]).
\end{equation*}
Since $G_{*}$ is an isomorphism,  the kernel of the left-hand side has dimension $2n-1$. In theory, $G_{*|\iota}$ might affect positivity in the remaining direction. However, when $\iota=\iota_0$ we can compute the symbol explicitly:
\begin{align*}
 \sigma(G_{*|\iota_0}(H_J[\iota_0]))_{|x}(\zeta)&=\sigma(-\iota_{0*}^{-1}(JH_J[\iota_0]))_{|x}(\zeta): T_xL\oplus T_xL\rightarrow T_xL\oplus T_xL,\\
 (Y_1,Y_2)&\mapsto \big(0,\overline{g}(J\iota_{0*}\zeta^\#, J\iota_{0*}Y_2)\zeta^\#\big)=g(\zeta^\#, Y_2)(0,\zeta^\#).
\end{align*}
This endomorphism has positive eigenvector with $(Y_1,Y_2)=(0,\zeta^\#)$, so it has positive trace. This is an open condition, so the same holds for $\sigma(G_{*|\iota}(H_J[\iota]))$, for any $\iota$ sufficiently $C^2$-close to $\iota_0$. Replacing $\iota$ with $F(X)$ gives the following.

\begin{lem}\label{l:symbol_tildeH}
 For any $X$ sufficiently $C^2$-small, the symbol of $X\mapsto \widetilde{H}_J[X]$ has a $(2n-1)$-dimensional kernel and one positive eigenvector.
\end{lem}

Up to here, $X$ could be any immersion $L\rightarrow \mathcal{U}$, but we now restrict to $X\in\Lambda^0(\mathcal{U})$, \textit{i.e.}~to sections of $TL$. Our main motivation for doing this is to apply \cite[Theorem 5.1]{Hamilton} to $H_J+T_J$; notice that it also corresponds to the idea of reducing degeneracies through gauge-fixing. In order for the flow to preserve sections, however, it will be necessary to project the operator onto the distribution in $T\mathcal{U}$ determined by the fibres of the vector bundle $TL$. 

We can also describe this space of sections in our original setting, given by maps into $M$. The map $F_*$ sends the distribution in $T\mathcal{U}$ to an integrable distribution $\mathcal{D}$, contained in $TM$. Sections of the vector bundle then correspond to the space of totally real  immersions
$$\mathcal{S}[\iota_0]:=\{\mbox{$\iota:L\rightarrow M$ such that $\iota_0(x)$, $\iota(x)$ belong to the same leaf of $\mathcal{D}$}\}.$$

In Hamilton's framework the next step would be to pull the integrability condition \eqref{eq:integrability} 
back to $\mathcal{U}$ and show it is satisfied by the projected operator. Given 
the above, however, it is  clear that we can equivalently continue to work in 
$\mathcal{V}$, applying Hamilton's result to our restricted space of 
immersions $\mathcal{S}[\iota_0]$. This will simplify some of the notation. 

\paragraph{Existence of Lagrangian solutions.}
The first problem is to define the projection of $H_J+T_J$ so as to preserve the integrability condition \eqref{eq:integrability}. We make use of the fact that $Z\mapsto L_{\iota}Z$ given in \eqref{eq:Liota} has a large kernel (reflected in its symbol), given by sections of $J(\iota_*(TL)^\perp)=(J\iota_*(TL))^{\perp}$. When $\iota=\iota_0$ this space is orthogonal to the distribution $\mathcal{D}$, so for immersions $C^1$-close to $\iota_0$ it is transverse to the distribution. Given any such $\iota$, we thus obtain a splitting 
$$T_{\iota(x)}M=\mathcal{D}_{\iota(x)}\oplus J(\iota_*(T_xL)^\perp).$$
Let $\pi$ denote the corresponding projection onto the second factor. Notice that the splitting, thus the projection, depends on first-order information in $\iota$. 

Set $K_J[\iota]:=\pi(H_J[\iota]+T_J[\iota])$. Then $H_J+T_J-K_J$ belongs to the distribution and continues to satisfy the integrability condition \eqref{eq:integrability}.  

We now show that this modification of the Maslov flow is well-posed.

\begin{thm}\label{thm:existence_projectedeq}
Let $(M,J,\bar\omega)$ be an almost K\"ahler manifold. Let $\iota_0:L\to M$ be a 
totally real immersion.
Then 
 $$\frac{\partial \iota_t}{\partial t}=H_J[\iota_t]+T_J[\iota_t]-K_J[\iota_t],\ \ \iota_t|_{t=0}=\iota_0$$
admits a unique short-time solution in $\mathcal{S}[\iota_0]$. 
 \end{thm}
\begin{proof}
 As already discussed, we linearize the setting ($t$-independently) for this problem by applying $G$: $\mathcal{S}[\iota_0]$ then corresponds to the sections of $TL$, setting us exactly in the correct framework for applying \cite[Theorem 5.1]{Hamilton}. We need only check that the positivity condition of that theorem is satisfied: in terms of $\mathcal{S}[\iota_0]$, this corresponds to proving that the symmetric endomorphism $\sigma(H_J[\iota]+T_J[\iota]-K_J[\iota])(\zeta)_{|\iota(x)}$ is positive for 
 $\iota$ near $\iota_0$ when restricted to 
$$\mbox{Ker}\left(\sigma(Z\mapsto L_\iota(Z))(\zeta)\right)\cap \mathcal{D}_{\iota(x)}.$$
When $\iota=\iota_0$
this space coincides with $\langle J\iota_*\zeta^\#\rangle$: it has dimension one and the positivity condition is fulfilled. For small perturbations of $\iota$
this dimension can only decrease. However, we know that it must always be at least $1$-dimensional  because $H_J+T_J-K_J$ satisfies the integrability condition. Furthermore, positivity is an open condition. This proves that  Hamilton's criterion holds, so we can apply \cite[Theorem 5.1]{Hamilton} to obtain the existence and uniqueness of the solution $\iota_t$ within the class $\mathcal{S}[\iota_0]$.
\end{proof}
\begin{cor}\label{Lag.MF.exist.cor}
  Assume $(M,J,\bar\omega)$ is almost K\"ahler and satisfies $\tilde{P}=2\lambda\bar{\omega}$ for some $\lambda\in\R$ (as defined in Section \ref{ss:SCF}) and that $\iota_0$ is Lagrangian. Then the Maslov flow
  $$\frac{\partial \iota_t}{\partial t}=H_J[\iota_t]+T_J[\iota_t],\ \ \iota_t|_{t=0}=\iota_0$$
 admits a unique short-time solution $\iota_t:L\to M$ in $\mathcal{S}[\iota_0]$, and it is Lagrangian.
 \end{cor}
 
\begin{proof}
Let $\iota_t$ be the solution obtained in Theorem \ref{thm:existence_projectedeq}.  The integrability  condition, together with the hypothesis $\tilde{P}=2\lambda\bar{\omega}$, shows that $\frac{\partial}{\partial t}\left(\iota_t^*\bar{\omega}\right)=\lambda\iota_t^*\bar{\omega}$, so $\iota_t$ is Lagrangian at each time as $\iota_0$ is Lagrangian. In turn, this implies that $K_J[\iota_t]$ is a tangent vector field, \textit{i.e.}
$$K_J[\iota_t]=(\iota_t)_*(X_t),$$
 for a curve of vector fields $X_t$ on $L$. 
 Let $\phi_t$ be the curve of diffeomorphisms of $L$ obtained by integrating $X_t$ and set $j_t(x):=\iota_t(\phi_t(x))$. Then
 \begin{align*}
 \frac{\partial j_t}{\partial t}&=\frac{\partial \iota_t}{\partial t}_{|\phi_t}+(\iota_t)_*\left(\frac{\partial \phi_t}{\partial t}\right)=H_J[\iota_t]_{|\phi_t}+T_J[\iota_t]_{|\phi_t}-K_J[\iota_t]_{|\phi_t}+(\iota_t)*(X_t)\\
 &=H_J[\iota_t\circ\phi_t]+T_J[\iota_t\circ\phi_t]=H_J[j_t]+T_J[j_t],
 \end{align*}
 solving the equation.
 \end{proof}
\begin{remark}
A key example of almost K\"ahler manifolds with $\tilde{P}=2\lambda\bar{\omega}$ are solitons for symplectic curvature flow, so the previous sections suggest that Corollary \ref{Lag.MF.exist.cor} is yet another manifestion of the interaction between Maslov flow and SCF. It would thus be interesting to extend Theorem \ref{thm:existence_projectedeq} and Corollary \ref{Lag.MF.exist.cor} to the (modified) coupled flow MF+SCF. Specifically, suppose the structure $(\bar{g}_t,J_t,\bar{\omega}_t)$ on $M$ moves by symplectic curvature flow (which we know has short-time existence) with initial condition $(\bar{g},J,\bar{\omega})$. In this case we could define $F$, $G$, $\mathcal{D}$ and $\mathcal{S}[\iota_0]$ as before, with respect to the fixed metric $\bar{g}$: this data serves only to linearize the setting, thus does not need to depend on $t$. Since the totally real condition is open, for $t$ sufficiently small the maps in $\mathcal{S}[\iota_0]$ are also $J_t$-totally real and we can repeat our construction to 
obtain a varying family of splittings
$$T_{\iota(x)}M=\mathcal{D}_{\iota(x)}\oplus J_t(\iota_*(T_xL)^{\perp_t})$$
for $\iota$ near $\iota_0$, thus $t$-dependent projections $\pi_t$ onto the second factor.  We now define $K_J^t[\iota]:=\pi_t(H_J^t[\iota]+T_J^t[\iota])$, where $H_J^t$ and $T_J^t$ are computed using the structure $(\bar{g}_t,J_t,\bar{\omega}_t)$, obtaining a modified Maslov flow.

The situation is thus very similar to the case above, but the existence proof given there fails because now the integrability operator $L$ depends explicitly on $t$ through $\bar\omega_t$: this situation is more complicated than that considered in \cite{Hamilton}, because of the extra terms generated by the $t$-derivative.
\end{remark}

\begin{remark}
A short-time existence result similar to Corollary \ref{Lag.MF.exist.cor} was observed 
in \cite{SmoWang} (when applied to the Chern connection).  
There the authors define a \emph{normal} vector field $\hat{H}$ on any
 totally real submanifold, which agrees up to first order terms with the 
 mean curvature vector $H$. The $\hat{H}$-flow is thus weakly parabolic, and short-time 
 existence follows from standard arguments. A maximum principle method, as 
 in Section \ref{review}, shows that the $\hat{H}$-flow preserves the 
 Lagrangian condition if $\tilde{P}=2\lambda\bar{\omega}$.
We observed in Section \ref{ss:otherflows} that, on Lagrangians, $\hat{H}=H_J+T_J$ and thus the $\hat{H}$-flow coincides with the Maslov flow on Lagrangians. This does not, however, imply that the operators have the same symbol: indeed, we have shown that the symbols are very different. On the one hand, the $\hat{H}$-flow has a much easier existence theory. On the other, it has no special geometric properties on generic totally real submanifolds: it is neither a gradient flow (like $J$-MCF), nor does it couple well with any ambient flow (like Maslov flow).
\end{remark}

\paragraph{Comparison with MCF and uniqueness.}It is useful to stress the analogies and differences with the corresponding proof for MCF. Assume we want to show 
$$\frac{\partial\iota}{\partial t}=H[\iota]$$
admits a solution, which is unique. Let us first restrict to normal variations. We use the initial immersion $\iota_0$ and the exponential map to build a local diffeomorphism defined on (a neighbourhood of the pull-back of) the normal bundle
$$F:\mathcal{U}\subseteq \iota_0^*(\iota_{0*}(TL)^\perp)\rightarrow\mathcal{V}\subseteq M$$ 
and restrict our attention to immersions obtained as sections of $\mathcal{U}$. Notice that, as above, the corresponding distribution $\mathcal{D}$ in $TM$ will be normal to $\iota_*(TL)$ only when $\iota=\iota_0$, so to preserve such sections we must project the equation onto $\mathcal{D}$. This implies that, as a first step, we try to solve the perturbed equation
$$\frac{\partial\iota}{\partial t}=H[\iota]-\pi(H[\iota]).$$
Once again, defining the projection requires the choice of a splitting of $T_{\iota(x)}M$. In this case it is convenient to make the choice
$$T_{\iota(x)}M=\mathcal{D}_{\iota(x)}\oplus \iota_*(TL),$$
because (i) the projected equation is parabolic, so we immediately obtain existence and uniqueness within the class of sections, and (ii) for any initial data, we find ourselves in the same situation as in Corollary \ref{Lag.MF.exist.cor}: since the perturbation is tangential, we can recover a solution to MCF via reparametrization.

It is important to emphasise, however, that for MCF any choice of projection  generates a parabolic equation when restricted to the space of sections: we can thus  make the choice that best allows for returning to the original equation. For Maslov flow, the additional degeneracies make us use the integrability condition to obtain existence, thus forcing us to carefully choose a projection which preserves this condition. This is the main difference between the two flows.

The final step is to prove uniqueness within the wider class of all immersions. In 
this regard the two equations are very similar. The space of sections is locally in 
1:1 correspondence with the space of non-parametrized submanifolds, 
\textit{i.e.}~with the space of immersions modulo reparametrization via diffeomorphisms of $L$. The above uniqueness result within the space of sections implies uniqueness of the corresponding flow in the space of non-parametrized submanifolds. Any two solutions $\iota_t$, $\iota_t'$ to MCF or to the Maslov flow, given the same initial data, thus define the same image submanifold $L_t\subseteq M$ at each time: they differ only by a 1-parameter family of reparametrizations, \textit{i.e.}~$\iota'_t=\iota_t\circ\phi_t$. It follows that the corresponding time derivatives differ only by a tangential term. On the other hand, the equations prescribe a motion which is transverse to the tangent space: orthogonal in the case of MCF, in $J(TL)$ in the case of the Maslov flow. In both 
cases, this implies that such tangential terms must vanish, so $\phi_t\equiv Id$.

 \section{Maslov form in Calabi--Yau manifolds}\label{s:maslov_CY}
 
The class of Ricci-flat KE manifolds contains a special subclass: Calabi--Yau manifolds. In this setting there is a classical notion of Maslov form, which we review in this section and compare with the Maslov form introduced above. This leads to a new characterization of the critical points of the $J$-volume functional.
 
 \subsection{Classical Maslov form}\label{ss:classical_maslov}
 \paragraph{Lagrangian Grassmannian.}Consider the Grassmannian of all oriented $n$-planes in $\C^n$. The Lie group $\U(n)$ acts on it by rotation. Let $\R^n$ denote the $n$-plane spanned by the standard vectors $\partial_{x_1},\dots,\partial_{x_n}$, with the corresponding orientation. The Grassmannian $\text{Lag}^+$ of oriented Lagrangian planes can then be described as the orbit of $\R^n$ under this action. The subgroup preserving $\R^n$ is $\SO(n)$: this shows that $\text{Lag}^+$ can be identified with the homogeneous space $\U(n)/\SO(n)$. Consider the map
 \begin{equation}
  e^{i\theta}:\text{Lag}^+\simeq \U(n)/\SO(n)\rightarrow \mathcal{S}^1,\ \ \pi\simeq [U]\mapsto \mbox{det}_{\C}(U).
 \end{equation}
This defines the \textit{Lagrangian angle} $\theta$ of the oriented Lagrangian plane, up to multiples of  $2\pi$.

Now let $V$ be any $n$-dimensional Hermitian vector space. Choose an isomorphism $\phi:V\rightarrow \C^n$ which identifies corresponding structures. We can then use $\phi$ to identify the oriented Lagrangian Grassmannian of $V$ with $\U(n)/\SO(n)$. Notice that $\phi$ is well-defined only up to left multiplication by $\U(n)$, so the identification of Grassmannians is also well-defined only up to left multiplication. This shows that, in this context, the Lagrangian angle is not well-defined. 

To obtain a Lagrangian angle for planes in $V$, we first observe that $\det_{\C}$ can be identified with the $n$-form $\d z:=\d z^1\wedge\dots\wedge \d z^n$ 
on $\C^n$. Let us thus assume that $V$ is further endowed with a complex $n$-form $\Omega$, which is equal to $\phi^*\d z$. Then the space of isomorphisms identifying all given structures is well-defined up to left multiplication by $\SU(n)$, so the Lagrangian angle is now well-defined.

\paragraph{Totally real Grassmannian.}Analogous considerations allow us to identify the Grassmannian $\text{TR}^+$ of oriented totally real $n$-planes in $\C^n$ with the homogeneous space $\GL(n,\C)/\GL^+(n,\R)$. Recall the standard Polar Decomposition theorem: any $M$ in $\GL(n,\C)$ has a unique decomposition $M=PU$, where $P$ is positive self-adjoint and $U\in\U(n)$. There is an explicit formula: $P=\sqrt{MM^*}$, thus $U=(\sqrt{MM^*})^{-1}M$. Consider the composition 
$$\begin{array}{rcccl}
 e^{i\theta}:\text{TR}^+\simeq \GL(n,\C)/\GL^+(n,\R) & \!\!\!\rightarrow\!\!& \U(n)/\SO(n) &\!\!\!\rightarrow\!\!\! &\mathcal{S}^1\\
 \pi\simeq [M] &\!\!\!\mapsto\!\!\!& [U] &\!\!\!\mapsto\!\!\!& \mbox{det}_{\C}(U)=\frac{\mbox{det}_{\C}(M)}{|\mbox{det}_{\C}(M)|}.
\end{array}$$
Again, this implicitly defines an angle $\theta(\pi)$; if $\pi$ is Lagrangian, it coincides with the Lagrangian angle.

Once again, in a Hermitian vector space $V$ the identification of the oriented totally real Grassmannian with $\GL(n,\C)/\GL^+(n,\R)$ is well-defined only up to left multiplication by $\U(n)$. The projection map $\GL(n,\C)\rightarrow \U(n)$ defined by polar decomposition is equivariant with respect to this multiplication. This implies that, as long as $V$ is further endowed with a $n$-form $\Omega$ as above, we can define the angle of any oriented totally real $n$-plane $\pi$. We can also calculate this angle intrinsically, as follows:
\begin{align}\label{eq:angle_calc}
 e^{i\theta(\pi)} =\Omega(e_1,\dots, e_n) &=\Omega(v_1,\dots,v_n)/|\Omega(v_1,\dots,v_n)|\nonumber\\
 &=\Omega(v_1,\dots,v_n)/|v_1\wedge\dots\wedge v_n|_h,
\end{align}
where $e_1,\dots,e_n$ is a positive orthonormal basis for the Lagrangian plane obtained from $\pi$ via polar decomposition, whilst $v_1,\dots,v_n$ is any
 positive oriented basis of $\pi$.

\paragraph{Classical Maslov form.}Now consider $\C^n$ as a manifold. The oriented Lagrangian Grassmannian is a trivial fibre bundle over $\C^n$; we can identify it with $\U(n)/\SO(n)\times\C^n$. Let $\iota:L\rightarrow\C^n$ be a Lagrangian immersion. The \textit{Lagrangian angle} of $L$ is the function $\theta_L$ on $L$ defined by $\theta_L(x):=\theta(T_xL)$. Consider the corresponding map
\begin{equation}
 e^{i\theta_L}:L\rightarrow \mathcal{S}^1.
\end{equation}
The \textit{classical Maslov form} is then the $1$-form $\mu_L:=(e^{i\theta_L})^*\d\theta=\d\theta_L$. Notice that this is a well-defined closed form on $L$, even though the underlying angle is only well-defined up to multiples of $2\pi$.

\paragraph{Extensions of the classical Maslov form.}It is simple to extend the definition of the classical Maslov form in two ways. 

First, assume we have a totally real immersion $\iota:L\rightarrow\C^n$. We can then define an angle function as above, setting $\theta_L(x):=\theta(T_xL)$. The \textit{classical Maslov form} is then $\mu_L:=(e^{i\theta_L})^*\d\theta=\d\theta_L$.

Second, assume $M$ is an almost Hermitian manifold endowed with a global non-zero smooth $(n,0)$-form $\Omega$, normalized to have length 1. In particular this 
implies that $K_M$ is differentiably trivial. For each $p\in M$, $V:=T_pM$ is a Hermitian vector space endowed with a form $\Omega[p]$, isomorphic to $\C^n$ with 
its standard structures. In this case the Grassmannians of oriented totally real and Lagrangian planes are not trivial bundles over $M$, but they still have standard 
fibres $\GL(n,\C)/\GL^+(n,\R)$ and $\U(n)/\SO(n)$, respectively. 

Let $\iota:L\rightarrow M$ be a totally real immersion. We can define the angle function pointwise, as above. It will be smooth (though well-defined only up to multiples 
of $2\pi$) because the data on $M$ is smooth. We then obtain a form $\mu_L$ as above.

\subsection{Comparison in Calabi--Yau manifolds}\label{ss:comparison}
In Section \ref{ss:canonical_data} we defined a notion of Maslov form $\xi_J$
for a totally real submanifold, valid in great generality: there, $M$ was any almost complex manifold endowed only with the additional structure of a Hermitian metric and connection on $K_M$. In Section \ref{ss:classical_maslov} we reviewed the classical definition of the Maslov form $\mu_L$. We now want to compare these two definitions.

Before proceeding we should notice the following closely related facts.
\begin{itemize}
 \item There is a discrepancy between the contexts of the two definitions: the classical Maslov form requires an almost Hermitian structure on $M$ plus a choice of $\Omega$. It therefore requires the topological restriction $c_1(M)=0$. It does not however require a connection on $K_M$. To compare the two definitions, we will need to choose a common setting. 
\item In the classical setting the Maslov form is always closed, though usually it is not exact because in general there will be no way to resolve the fact that $\theta_L$ is well-defined only up to multiples of $2\pi$. It thus defines a cohomology class on $L$, known as the \textit{Maslov class}. Notice that the class of $\d\theta$ is an integral class in the cohomology of $\mathcal{S}^1$ (at least up to normalization). This implies that the Maslov class is also integral. 

In our setting Proposition \ref{xiJconn.prop} shows that in general $\xi_J$ is not closed, so there is no notion of Maslov class. In particular, integrating $\xi_J$ along a loop in $M$ depends on the specific curve, so it will not yield integral values. Proposition \ref{xiJconn.prop} shows however that $\xi_J$ will be closed if $c_1(M)=0$.
\end{itemize}

The previous comments show that we should concentrate on manifolds for which $c_1(M)=0$. The most important such class is the following.

\begin{dfn}
 A K\"ahler manifold $(M,\overline{g},J,\overline{\omega})$ is \textit{Calabi--Yau} (CY) if it is endowed with a global section $\Omega$ of $K_M$, parallel with respect to the Levi-Civita connection and normalized to have constant length $1$. 
 \end{dfn}
 
 The existence of a parallel tensor implies a reduction of the holonomy group of $(M,\overline{g})$. In this case, the holonomy group is contained in $\SU(n)$, so the metric is Ricci-flat. Furthermore $\Omega$ is holomorphic, so $K_M$ is holomorphically trivial. 

\begin{lem}\label{l:same_angle}
Let $M$ be an almost Hermitian manifold endowed with a global non-zero $(n,0)$-form $\Omega$, normalized to have length 1. Let $\iota:L\rightarrow M$ be a totally real submanifold with angle function $\theta_L$. Let $\Omega_J$ denote the canonical section of $K_M[\iota]$, as in Section \ref{ss:canonical_data}. Then
 \begin{equation*}
  \Omega=e^{i\theta_L}\Omega_J.
 \end{equation*}
\end{lem}
 \begin{proof}
 Choose $x\in L$ and let $v_1,\dots,v_n$ be an oriented basis of $T_xL$. It suffices to prove that 
 $$\Omega(v_1,\dots,v_n)=e^{i\theta_L}\Omega_J(v_1,\dots,v_n).$$
 Recall that $\Omega_J=\frac{v_1^*\wedge\dots\wedge v_n^*}{|v_1^*\wedge\dots\wedge v_n^*|_h}$.
Since
 $$|v_1\wedge\dots\wedge v_n|_h\cdot|v_1^*\wedge\dots\wedge v_n^*|_h=1,$$
the result then follows from (\ref{eq:angle_calc}).
 \end{proof}
\begin{remark}
 Since both $\Omega$ and $\Omega_J$ have length 1, they are related by some angle function. The above lemma shows that this angle function is precisely $\theta_L$.
\end{remark}

\begin{prop}
 Let $M$ be a CY manifold and $\iota:L\rightarrow M$ be totally real.
 Then $\xi_J=-\mu_L$.
\end{prop}
\begin{proof}
 Since $\Omega$ is parallel, Lemma \ref{l:same_angle} shows that
 $$0=\overline{\nabla}\Omega=\d (e^{i\theta_L})\otimes\Omega_J+e^{i\theta_L}\overline{\nabla}\Omega_J.$$
 Thus
 $$i\xi_J \otimes\Omega_J=\overline{\nabla}\Omega_J=-\frac{\d (e^{i\theta_L})}{e^{i\theta_L}}\otimes\Omega_J=-\d(\log e^{i\theta_L})\otimes\Omega_J=-i\d\theta_L \otimes\Omega_J.$$
 Since $\mu_L=\d\theta_L$, the result follows.
\end{proof}

Using Proposition \ref{HJ-prop} leads to the following, interesting, characterization.

\begin{cor}\label{cor:minimal_angle}
 Let $M$ be a CY manifold and $\iota:L\rightarrow M$ be totally real. Then $H_J=J\nabla \theta_L$, so $L$ is a critical point of the $J$-volume functional if and only if $\theta_L$ is constant.
\end{cor}

\section{The Calabi--Yau calibration}\label{s:calibration}

Recall the standard setting of calibrated geometry. We start with a Riemannian manifold $(M,\overline{g})$. A differential $k$-form 
$\alpha$ on $M$ is a \textit{calibration} if:
\begin{itemize}
 \item it is closed, \textit{i.e.}~$\d\alpha=0$;
 \item it is bounded by the Riemannian volume in the following sense. Let $\text{Gr}^+(k,M)$ denote the Grassmannian bundle of oriented $k$-planes in $M$. We ask that, for any $\pi\in \text{Gr}^+(k,M)$,
\begin{equation}\label{comass.eq}
\alpha_{|\pi}\leq \vol_{g}[\pi],
\end{equation}
where $\vol_{g}[\pi]$ denotes the induced volume form.
 \end{itemize}
 It is simple to check that $-\alpha$ is also a calibration.
 
A $k$-dimensional oriented submanifold $\iota:L\rightarrow M$ is \textit{calibrated}  (by $\alpha$) if it achieves the equality: $\iota^*\alpha\equiv \vol_g$. In this case a simple computation (cf.~Lemma \ref{STRmin.lem} below) shows that $L$ is volume-minimizing in its oriented homology class. Notice that the same submanifold with the opposite orientation is then calibrated by $-\alpha$.

\medskip

Calabi--Yau manifolds $(M,\overline{g},J,\overline{\omega},\Omega)$ are a well-known example. Since $\Omega$ is parallel, $\Ree(\Omega)$ is closed.  Since $|\Omega|_h\equiv 1$, we have  
$|\Omega(e_1,\dots,e_n)|\leq 1$ for orthonormal vectors $e_1,\ldots,e_n$. Hence $|\Ree(\Omega)(e_1,\dots,e_n)|\leq 1$, so
\begin{equation}\label{eq:CY_ineq}
 \Ree(\Omega)_{|\pi}\leq \mbox{vol}_{g}[\pi].
\end{equation}
The equality condition $|\Omega(e_1,\dots,e_n)|=1$ is particularly interesting: one can check that it is equivalent to the condition that the plane $\pi$ generated by these vectors is Lagrangian, thus providing a new characterization of Lagrangian planes. 
Splitting $\Omega$ into real and imaginary parts, it follows that if equality holds in \eq{eq:CY_ineq} then 
$\pi$ must be Lagrangian and $\Imm(\Omega)_{|\pi}\equiv 0$. The converse also holds, 
thus characterizing the submanifolds $\iota:L\rightarrow M$ calibrated by $\pm\Ree(\Omega)$ as those for which 
$\iota^*\overline{\omega}\equiv 0$ and $\iota^*\Imm(\Omega)\equiv 0$. 

More generally, for any fixed $\theta$, $e^{i\theta}\Omega$ has the same properties as $\Omega$ 
so we get an $\mathcal{S}^1$-family of calibrations on $M$. The submanifolds calibrated by any $\Ree(e^{i\theta}\Omega)$ 
are called \textit{special Lagrangian} (SL). We can characterize them as follows.
\begin{lem}\label{l:equiv_SL}
 Let $M$ be a CY manifold and $\iota:L\rightarrow M$ be an immersion with $L$ connected. The following are equivalent characterizations of the SL condition:
 \begin{itemize}
  \item[(a)] $\iota^*\Ree(e^{i\theta}\Omega)\equiv \vol_g$ (for some $\theta$);
  \item[(b)] $\iota^*\overline{\omega}\equiv 0$ and $\iota^*\Imm(e^{i\theta}\Omega)\equiv 0$ (for some $\theta$);
  \item[(c)] $\iota^*\overline{\omega}\equiv 0$ and the Lagrangian angle $\theta_L$ is constant;
  \item[(d)] $L$ is minimal Lagrangian.
 \end{itemize}
\end{lem}
\begin{proof}
The SL property is defined by (a). The equivalence of (a)-(c) follows from the definitions. 
The equivalence with (d) follows from Corollary \ref{cor:minimal_angle}.
\end{proof}

\subsection{{\boldmath $J$}-volume and the CY calibration}\label{ss:CY_calibration}
Recall that in Section \ref{ss:canonical_data} we associated to any oriented totally real $n$-plane $\pi$ the $n$-form $\vol_J[\pi]$. We can use it to strengthen (\ref{eq:CY_ineq}), thus decoupling the two conditions in Lemma \ref{l:equiv_SL}(b), as follows.

\begin{lem}\label{Omega.ineq.lem}
Let $M$ be an almost Hermitian manifold endowed with a global non-zero $(n,0)$-form $\Omega$, normalized to have length 1. Fix $\pi\in \text{\emph{TR}}^+$. Then 
$$\Ree(\Omega)_{|\pi}\leq \vol_J[\pi]\leq \vol_g[\pi].$$
Furthermore, we have that equality holds
\begin{itemize}
\item in the first relation if and only if $\Imm(\Omega)_{|\pi}=0$ and $\Ree(\Omega)_{|\pi}>0$;
\item and in the second relation if and only if $\pi$ is Lagrangian.
\end{itemize}
\end{lem}
\begin{proof}
Choose a positively oriented basis $v_1,\dots,v_n$ of $\pi$. We need to prove that 
$$\Ree(\Omega)(v_1,\dots,v_n)\leq \vol_J[\pi](v_1,\dots,v_n)\leq \vol_g[\pi](v_1,\dots,v_n).$$
This statement concerns numbers obtained by evaluating $n$-forms. Recall that $\vol_J$ is simply the restriction of the canonical section $\Omega_J$ and, cf.~Lemma \ref{l:same_angle}, that $\Omega=e^{i\theta_L}\Omega_J$. Thus the absolute values satisfy:
\begin{equation*}
|\Ree(\Omega)(v_1,\dots,v_n)|\leq |\Omega(v_1,\dots,v_n)|=
|\vol_J(v_1,\dots,v_n)|.
\end{equation*}
The result now follows.
\end{proof}

\begin{remark}
Notice that the inequalities in Lemma \ref{Omega.ineq.lem} hold trivially for oriented $n$-planes $\pi\notin \text{TR}^+$, \text{i.e.}~partially complex $n$-planes,
since $\Ree(\Omega)_{|\pi}=\vol_J[\pi]=0$.
\end{remark}

Let us apply Lemma \ref{Omega.ineq.lem} when $M$ is CY. We then find that 
$$\d\Ree(\Omega)=0,\ \ \Ree(\Omega)_{|\pi}\leq \vol_J[\pi].$$
We will say that $\Ree(\Omega)$ is a \textit{calibration} on $M$ \textit{tamed} by $\vol_J$.  As before $\Ree(e^{i\theta}\Omega)$ also gives a calibration on $M$ 
tamed by $\vol_J$ for any constant $\theta$.

\begin{remark}
The standard definition of calibration uses $\vol_g$, thus a Riemannian structure on $M$. It is interesting to notice that $\vol_J$ is defined using only complex data: $J$, $\Omega$ and $h$ on $K_M$. It follows that this notion of calibration does not need a Riemannian structure on $M$.
\end{remark}

We will say that an oriented totally real submanifold $L$ is $J$-\textit{calibrated} if $\iota^*\Ree(\Omega)=\vol_J$. We say that it is \textit{special totally real} (STR) if it is $J$-calibrated by $\Ree(e^{i\theta}\Omega)$, for some $\theta$.

Using these definitions we have the following analogue of Lemma \ref{l:equiv_SL}.

\begin{lem}\label{l:equiv_STR}
 Let $M$ be a CY manifold and $\iota:L\rightarrow M$ be an immersion with $L$ connected. The following are equivalent characterizations of the STR condition:
 \begin{itemize}
  \item[(a)] $\iota^*\Ree(e^{i\theta}\Omega)\equiv \vol_J$ (for some $\theta$);
  \item[(b)] $\iota^*\Imm(e^{i\theta}\Omega)\equiv 0$ (for some $\theta$);
  \item[(c)] the angle $\theta_L$ is constant;
  \item[(d)] $L$ is a critical point for the $J$-volume.
 \end{itemize}
\end{lem}

\begin{proof}
Condition (a) is the definition of STR.  The equivalence of (a) and (b) is a consequence of Lemma \ref{Omega.ineq.lem}.  Lemma 
\ref{l:same_angle} implies the equivalence of (b) and (c).  The equivalence of (c) and (d) is Corollary \ref{cor:minimal_angle}.
\end{proof}

\begin{remark}
STR submanifolds were introduced in the special case of $\C^n$ in \cite{Bor}. Examples are given therein of STR submanifolds which are not SL.
\end{remark}

As usual in the context of calibrations, our interest in this class of submanifolds lies in the following calculation. 

\begin{lem}\label{STRmin.lem}
Compact STR submanifolds minimize the $J$-volume in their homology class. 
Furthermore, if $L$ is compact STR and $\Vol_J(L')=\Vol_J(L)$ for some other $L'\in [L]$, then $L'$ is also STR.
\end{lem}

\begin{proof}
Assume $L$ is compact and STR.   Then, for any compact $L'$ homologous to $L$,
\begin{equation*}
 \int_L\vol_J=\int_L\Ree(e^{i\theta}\Omega)=\int_{L'}\Ree(e^{i\theta}\Omega)\leq\int_{L'}\vol_J.
\end{equation*}
The result follows.
\end{proof}

\begin{remark}
One can also define STR submanifolds in the weaker setting of almost Hermitian manifolds endowed with a global non-zero $(n,0)$-form $\Omega$, normalized to have length 1 and such that $\Ree(\Omega)$ is closed. However, since STR submanifolds have $\Imm(e^{i\theta}\Omega)|_L=0$ we see that there is an obstruction to the local existence of STR submanifolds if $\d\Imm(\Omega)\neq 0$.  
Hence, it is most natural to study STR submanifolds in the situation where $\Omega$ is closed, which forces $J$ to be integrable and $\Omega$ to be 
holomorphic.  Hence, $M$ must be Calabi--Yau.
\end{remark}

Lemma \ref{STRmin.lem} states that STR submanifolds  minimize $\Vol_J$.  However, partially complex submanifolds have $\Vol_J=0$ so we deduce the following.  

\begin{prop}\label{dichom.prop}
Let $\iota:L\rightarrow M$ be a compact totally real submanifold in a Calabi--Yau manifold $M$.  Then the following statements are mutually exclusive:
\begin{itemize}
\item[(a)] $[L]$ contains an STR submanifold;
\item[(b)] $[L]$ contains a partially complex submanifold.
\end{itemize} 
\end{prop}

Notice that the above follows alternatively from the fact that, on a partially complex submanifold, $\Omega$ and thus $\Ree(e^{i\theta}\Omega)$ vanishes. 

The following corollary refers to a particular subclass of CY manifolds: those which are hyperk\"ahler, \textit{i.e.}~have holonomy contained in $\Sp(n)\subset \SU(2n)$.

\begin{cor}
\begin{itemize}
\item[(a)] Let $L$ be a compact special Lagrangian in a Calabi--Yau manifold.  Then $L$ is not homologous to
a partially complex submanifold.
\item[(b)] Let $L$ be a compact complex submanifold of (complex) dimension $n$ in a hyperk\"ahler $4n$-manifold.  Then $L$ is not 
homologous to a special Lagrangian.
\end{itemize}
\end{cor}

The first of these two statements has an interesting consequence.

\begin{prop}\label{prop:unstable}
Let $L$ be a compact oriented Lagrangian in a Calabi--Yau manifold $M$.  If $[L]$ contains a partially complex submanifold, then Lagrangian 
mean curvature flow starting at $L$ cannot converge to a special Lagrangian.
\end{prop}

\subsection{Possible developments}\label{ss:CYdevelopments}

Let $\iota:L\rightarrow M$ be a Lagrangian submanifold in a CY manifold $M$. Clearly, $[\iota^*\overline{\omega}]=0\in H^2(L;\R)$. Let us also assume  $[\iota^*\Imm\Omega]=0\in H^n(L;\R)$. Ideally, under MCF such a Lagrangian will converge to an SL submanifold $\iota_\infty:L\rightarrow M$. 

In practice however this will generally not happen, both because of the development of singularities and because of possible further obstructions to the existence of an SL in the given homology class. The first of these issues might be solved using Lagrangian surgery near the singularity to create a new smooth  Lagrangian from which to restart the flow: notice that this could however change the topology of $L$. The second issue is currently still mysterious, and is conjectured to be related to some notion of ``stability'' of the given homology class. 

The simplest example of such  a stability condition is a consequence of Lemma \ref{l:equiv_SL}(a), which implies the need for yet another initial homological assumption:  $\int_
L\Ree\Omega>0$. A more elaborate notion of stability, motivated by Mirror Symmetry, appears in \cite{Thomas}; in \cite{ThomasYau} this is conjectured to be related to the long time existence of Lagrangian MCF, with convergence to an SL submanifold. Further notions of stability for Lagrangians and relations with Lagrangian MCF are discussed in \cite{Joy}.

It is interesting to speculate whether the Maslov flow (or equivalently $J$-mean curvature flow) of totally real submanifolds can be useful towards this programme. The following idea indicates a possible link between Lagrangian MCF and the notion of stability for SLs, the Maslov flow and STR submanifolds, and the work of Donaldson \cite{Don}.

Assume the initial Lagrangian $\iota$ develops a singularity under MCF. Rather than using surgery to restart the flow we could try to bypass this problem by replacing the initial condition $\iota$ with a perturbed immersion $\iota':L\rightarrow M$. This immersion would be totally real, and under the Maslov flow would ideally converge to an STR. Of course in general the Maslov flow might also develop singularities; however, one might at least hope that any ``non-essential'' singularity, arising from a bad choice of $\iota'$ rather than from intrinsic issues and thus not requiring a topology change, might be avoided via a generic choice of $\iota'$. 

The key point here would be the fact that, by relaxing the initial geometric assumptions on the immersion from Lagrangian to totally real, we would gain access to a much larger class of ``generic'' initial data. Notice that in this context our initial assumption $\int_L\Ree\Omega>0$ can be seen as a manifestation of Lemma \ref{l:equiv_STR} 
and Proposition \ref{dichom.prop}. In this process we would hope that the evolution of $\iota'$ does not stray too far from the evolution of $\iota$: the fact that the tensor $\omega$ is preserved under the flow may be some indication of this.

Let us thus assume that $\iota'$ has converged to an STR immersion $\iota''$. We now need a second geometric flow evolving $\iota''$ towards an SL. It turns out that such a flow exists: it is discussed in \cite{Don}. The notion of STR does not appear there but fits in nicely. We summarize the idea as follows.
\begin{itemize}
 \item The space of immersions $\mathcal{M}:=\{\iota:L\rightarrow M\}$ can be viewed as an infinite-dimensional manifold; its tangent space at $\iota$ is $T_\iota\mathcal{M}:=\Lambda^0(\iota^*TM)$, and thus inherits a complex structure from $M$. The integrability of $J$ on $M$ implies $\mathcal{M}$ is formally a holomorphic manifold.
 \item A priori $\mathcal{M}$ has no symplectic structure. Any choice of volume form $\sigma$ on $L$ will however induce a symplectic form $\overline{\sigma}$ on $\mathcal{M}$ defined as follows:
 $$\overline\sigma[\iota](X,Y):=\int_L\iota^*\overline{\omega}(X,Y)\sigma.$$
 This form is formally closed and compatible with the complex structure, so now $\mathcal{M}$ is K\"ahler.
 \item Choose a complex volume form $\Omega_0$ on $L$. Consider the subset $$\mathcal{M}_{\Omega_0}:=\{\iota:L\rightarrow M: \iota^*\Omega=\Omega_0\}.$$ A vector $X\in T_\iota\mathcal{M}$ is tangent to $\mathcal{M}_{\Omega_0}$ if and only if $\iota^*\mathcal{L}_X\Omega=0$; this is equivalent to $\iota^*\mathcal{L}_{JX}\Omega=0$, showing that $\mathcal{M}_{\Omega_0}$ is a complex submanifold of $\mathcal{M}$, thus K\"ahler. In particular, if $\Omega_0=\sigma$ is a real form then $\iota^*\Imm\Omega=0$ so $\mathcal{M}_{\sigma}$ is a subset of the space of STR immersions defined by the pointwise volume constraint $\iota^*\Ree(\Omega)=\sigma$. 
 \item As above, $\mathcal{M}_\sigma$ is K\"ahler. There is a natural action of the space of volume-preserving diffeomorphisms $G:=\Diff(L;\sigma)$, given by reparametrization of the immersions. Donaldson shows that this action is Hamiltonian, thus endowed with a moment map $\mu:\mathcal{M}_\sigma\rightarrow Lie(G)^*$. The zero set of this moment map is the space of Lagrangian submanifolds in $\mathcal{M}_\sigma$, thus SLs. Donaldson suggests that one might use the negative gradient flow of $|\mu|^2$ to locate SLs. This frames the existence problem of SLs into a standard setting for stability conditions: we refer to \cite{Don} for details.
\end{itemize}
Our idea is now clear: given the first limiting STR immersion $\iota''$, we choose $\sigma$ to be the induced $J$-volume form on $L$. This defines our space $\mathcal{M}_\sigma$ and the corresponding moment map; we can run the gradient flow of $|\mu|^2$, with initial data $\iota''$, to find a SL. This will live in the initially given homology class.

It is currently not clear if replacing Lagrangian MCF with this combination of Maslov flow and moment map flow really represents a technical improvement; however, it does serve to indicate a circle of ideas and possible relationships.

\paragraph{Moduli spaces.} It is well-known, cf.~\cite{Pacini} for a comprehensive review, that SL submanifolds generate smooth finite-dimensional moduli spaces. This is a direct consequence of the ellipticity of the coupled conditions  $\iota^*\overline{\omega}\equiv 0$, $\iota^*\Imm(e^{i\theta}\Omega)\equiv 0$. The condition defining STR submanifolds is not elliptic, so moduli spaces will not be finite-dimensional. The fact that the Maslov flow preserves $\omega=\iota^*\overline{\omega}$ indicates that it may be interesting to couple the STR condition with other conditions on $\omega$ and to study the corresponding moduli spaces. 

It is not known if SL moduli spaces are connected within a given homology class. 
It may be that the convexity property of the $J$-volume, and the existence of ``larger'' moduli spaces of STR submanifolds, will play a role in this 
direction. In particular, notice the following construction. 

Suppose we start with a special Lagrangian $\iota:L\rightarrow M$ 
and perturb it slightly to become totally real $\iota':L\rightarrow M$, but not
 Lagrangian. The stability of critical points makes us expect that the flow will exist for all time and converge to a smooth submanifold. 
 However, Proposition \ref{limit.prop} shows that this submanifold will not be Lagrangian; it will also not be partially complex 
 by Proposition \ref{dichom.prop}. We thus expect that $L'$ will flow 
to an STR submanifold.  Applying this construction to a family $\iota'_t$ of initial submanifolds converging to $\iota$, we construct a limiting family of 
STR submanifolds converging to $\iota$. This suggests that any SL submanifold should arise as a limit of a family of STR submanifolds.

\paragraph{Graphs.} Suppose we have symplectic manifolds $(M_1,\omega_1)$ and $(M_2,\omega_2)$ and we define
 $(M,\omega)$ to be the symplectic manifold with $M=M_1\times M_2$ and $\omega=\pi_1^*\omega_1-\pi_2^*\omega_2$, where $\pi_j:M\to M_j$ are the obvious projections.  Then, as is well-known, 
 the graph of a map $F:(M_1,\omega_1)\to (M_2,\omega_2)$ in $(M,\omega)$ is Lagrangian if and only if $F^*\omega_2=\omega_1$. This observation is used in \cite{MedosWang} to study the problem of deforming symplectomorphisms via Lagrangian mean curvature flow.
 
The totally real analogue of this situation is to consider almost complex manifolds $(M_1,J_1)$ and $(M_2,J_2)$ and define $(M,J)$ to 
be the almost complex manifold with $M=M_1\times M_2$ and 
$J=(J_1,-J_2)$.  If we have a map $F:(M_1,J_1)\to (M_2,J_2)$ then its graph is totally real if and only if whenever
$$J(X,F_*(X))=(Y,F_*(Y))$$
for tangent vectors $X,Y$ on $M_1$ then $X=Y=0$.  This equation becomes
$$J_1(X)=Y\quad\text{and}\quad -J_2\circ F_*(X)=F_*(Y)=F_*\circ J_1(X),$$
and so the graph is totally real if and only if 
$$J_2\circ F_*+F_*\circ J_1$$
is injective.  This is clearly an open condition, as one would expect.  The Maslov flow and $J$-mean curvature flow thus give tools for 
studying this space of maps.

In the case when $F:\C^n\to\C^n$ (or even from $\R^n$ to itself), we see that its graph is STR, and thus a critical point for the flow, if and only if 
$$\Imm \,\text{det}_{\C}(I+iF_*)=0\quad\text{and}\quad \Ree \,\text{det}_{\C}(I+iF_*)>0.$$
These are the same equations as for special Lagrangian graphs, except here we remove the condition that $F$ is given by the gradient of a scalar function.

\bibliographystyle{amsplain}
\bibliography{CFCLagr}

\end{document}